\documentclass[a4paper]{article}

\usepackage{amssymb}
\usepackage{amsmath}
\usepackage{mathtools}
\usepackage[amsmath,thmmarks,hyperref]{ntheorem}
\usepackage{enumerate}
\usepackage{color}
\usepackage{cite}

\usepackage{booktabs}

\usepackage{geometry}
\usepackage[final,notref]{showkeys}%
\usepackage[bookmarksnumbered=true,hidelinks]{hyperref}
\usepackage{caption}

\usepackage{relsize}

\theoremnumbering{arabic}
\theoremstyle{break}
\RequirePackage{latexsym}
\theorembodyfont{\itshape}
\theoremsymbol{}
\theoremheaderfont{\normalfont\bfseries}
\theoremseparator{}
\newtheorem{theorem}{Theorem}[section]

\newtheorem{proposition}[theorem]{Proposition}
\newtheorem{corollary}[theorem]{Corollary}

\theorembodyfont{\upshape}
\theoremsymbol{\ensuremath{\clubsuit}}
\newtheorem{remark}[theorem]{Remark}

\theorembodyfont{\upshape}
\theoremsymbol{}
\newtheorem{example}[theorem]{Example}

\theoremstyle{nonumberplain}
\theoremheaderfont{\normalfont\bfseries}
\theoremseparator{:}
\theorembodyfont{\normalfont}
\theoremsymbol{\ensuremath{\square}}
\RequirePackage{amssymb}

        \newtheorem{proof}{Proof}
\qedsymbol{\ensuremath{\bigsquare}}

\numberwithin{equation}{section}

\newcommand{\RR}{\mathbb{R}}
\newcommand{\NN}{\mathbb{N}}
\newcommand{\ZZ}{\mathbb{Z}}

\renewcommand{\d}{\mathrm{d}}

\newcommand{\vtd}[2]{{\mathbf{VTD}_{#2}^{#1}}}
\newcommand{\Q}[2]{{Q_{#2}^{#1}}}
\newcommand{\Qhat}[2]{{\widehat{Q}_{#2}^{#1}}}
\newcommand{\I}[2]{{\mathcal{I}_{#2}^{#1}}}
\newcommand{\Ihat}[2]{{\widehat{\mathcal{I}}_{#2}^{#1}}}
\newcommand{\Y}[1]{Y_{#1}}

\newcommand{\Pint}{\mathcal{P}}
\newcommand{\Qun}{Q_n}
\newcommand{\Id}{\mathrm{Id}}
\newcommand{\If}{\mathcal{I}}

\usepackage{mathrsfs}
\newcommand{\II}{
\mathchoice
{
\raisebox{-0.6mm}{$\displaystyle \mathlarger{\mathlarger{{\mathscr{I}}}}_{\!\!n}$}
}
{
\raisebox{-0.3mm}{$\mathlarger{\mathlarger{{\mathscr{I}}}}_{\!\!n}$}
}
{
\mathscr{I}_n
}
{
\mathscr{I}_n
}
}
\newcommand{\IIn}[1]{\II\!\left[#1\right]
}

\title{Variational Time Discretizations of Higher Order and Higher Regularity}
\author{Simon Becher\footnote{Institute of Numerical Mathematics, 
Technical University of Dresden, 01062 Dresden, Germany.
\mbox{e-mail: Simon.Becher@tu-dresden.de, Gunar.Matthies@tu-dresden.de}}
\and Gunar Matthies\footnotemark[1]}
\date{\today}

\begin{document}

\maketitle

\begin{abstract}
We consider a family of variational time discretizations that are
generalizations of discontinuous Galerkin (dG) and continuous
Galerkin--Petrov (cGP) methods. The family is characterized by two
parameters. One describes the polynomial ansatz order while the other one
is associated with the global smoothness that is ensured by higher order
collocation conditions at both ends of the subintervals. The
presented methods provide the same stability properties as dG or cGP.
Provided that suitable quadrature rules of Hermite type for evaluating the
integrals in the variational conditions are used, the  variational time
discretization methods are connected to special collocation methods. For
this case, we will present error estimates, numerical experiments, and
a computationally cheap postprocessing that allows to increase both the
accuracy and the global smoothness by one order.
\end{abstract}



\section{Introduction}
One way to solve parabolic partial differential equations starts with a
semi-discretization in space to obtain a huge system of ordinary
differential equations that is handled by a suitable temporal
discretization. If the spatial discretization gets finer, the system of
ordinary differential equations becomes stiffer. Hence, implicit methods
are preferable in order to exclude upper bounds for the time step length.
Moreover, the used time discretization should be at least A-stable to
ensure suitable stability properties. This means that the order of BDF
methods can be at most two. In order to have A-stable temporal
discretizations of higher order, implicit Runge--Kutta methods,
discontinuous Galerkin (dG), or continuous Galerkin--Petrov (cGP) schemes
could be applied.

This paper deals with a family of variational time discretizations that
generalizes dG and cGP methods. In addition to variational equations,
collocation conditions will be used. The considered family is characterized
by two parameters: the first one is the local polynomial ansatz order and
the second parameter is related to the global smoothness of the numerical
solution that is ensured by higher order collocation conditions at both
ends of the subintervals. 
With respect to their stability behavior, the family of variational
time discretizations can be divided into two groups. While the first group
shares its stability properties with the cGP method, the second group behaves
like the dG method. These observations, considered in~\cite{BMW19}, suggest
that the whole family of methods is appropriate to handle stiff problems.
Since the test space for each family member is
allowed to be discontinuous with respect to the time mesh, the discrete
problem can be solved in a time-marching process, i.e., by a sequence of
local problems on the subintervals.

Our studies can be build on a broad knowledge base on various
aspects of dG and cGP time discretization methods in the literature.
The a priori and a posteriori analysis of dG methods is well understood,
see~\cite{Tho06, EJ96}. Moreover, dG methods are known to be strongly
A-stable. Investigations of cGP($r$)-method for $r=1$ and systems of linear
ordinary differential equations can be found in~\cite[Sect.~9.3]{EJ96}.
Optimal error estimates and superconvergence results for the fully
discretized heat equation, based on a finite element method in space and a
cGP($r$) method in time, are given in~\cite{AM89}. The cGP($r$) method was
analyzed in~\cite{Sch10} for the affine linear case in an abstract Hilbert-space
setting and a general nonlinear system of ordinary differential equations
in the $d$-dimensional Euclidean space. Using energy arguments, the
A-stability of cGP($r$) methods was shown. In addition, cGP($r$) provides
an energy decreasing property for the gradient flow equation of an energy
functional, see~\cite{Sch10}.

Postprocessing techniques for dG and cGP methods applied to systems of
ordinary differential equations have been given in~\cite{MS11}. They allow
to obtain an improved convergence by one order in integral-based norms.
Furthermore, the postprocessed dG-solution will be continuous while the
postprocessing of cGP-solutions leads to continuously differentiable
trajectories.
This paper transfers and generalizes the postprocessing ideas to the whole
family of variational time discretizations. The postprocessing creates
an improved solution where the global smoothness is increased by one
differentiation order. Moreover, the postprocessing lifts the originally
obtained numerical solution on each time subinterval to the polynomial
space with one degree higher. This results in an increased accuracy in
integral-based norms,
see~\cite{MS11,ABM17,ES16,HST12,HST13,BKRS19,AM15,AM16}.
Please note that the postprocessing comes with almost no computational
costs since just jumps of derivatives of the discrete solution are needed.
Beside the improvements of accuracy and global smoothness, the
postprocessing can be used to drive an efficient adaptive time step
control, see~\cite{AJ15}.

As mentioned above, the variational time discretizations analyzed in this
paper use collocation conditions at the end points of the time
subintervals. We will show connections between pure collocation methods and
numerically integrated variational time discretizations provided a suitable
quadrature rule of Hermite type is applied. Based on these connections the
existence and uniqueness of discrete solutions will be shown. Moreover,
optimal error estimates follow. The connection between collocation methods
and postprocessed numerically integrated discontinuous Galerkin methods
using the right-sided Gauss--Radau quadrature formulas was considered
in~\cite{VR15}. Moreover, connections between collocation methods and the
numerically integrated continuous Galerkin--Petrov methods (using
interpolatory quadrature formulas with as many quadrature points as number
of independent variational conditions) are shown in~\cite{Hul72b,Hul72a}.

For affine linear systems of ordinary differential equations with
time-independent coefficients, an interpolation cascade is presented that
allows multiple postprocessing steps leading to very accurate solutions
with low computational costs. Moreover, temporal derivatives of discrete
solutions to affine linear systems with time-independent coefficients form
also solutions of variational time discretization schemes. This relation
was used to prove optimal error estimates for stabilized finite element
methods for linear first-order partial differential equations~\cite{ES16}
and for parabolic wave equations~\cite{BKRS19,ABBM19}.

The paper is organized as follows. Section~\ref{sec:methForm} provides some
notation and formulates the family of variational time discretizations. The
general postprocessing technique is considered in
Section~\ref{sec:postprocessing}. The connection between collocation methods
and numerically integrated variational time discretizations will be given
in Section~\ref{sec:collocation}. This connection is exploited to provide
results on the existence of unique solutions and to obtain error estimates.
We present in Section~\ref{sec:intCascade} for affine linear ODE systems an
interpolation cascade that allows multiple postprocessing steps. Properties
of derivatives of solutions to variational time discretization methods will
be discussed in Section~\ref{sec:nestedSol_Der}. Numerical experiments
supporting the theoretical results are presented in Section~\ref{sec:NumExp}.

\section{Notation and formulation of the methods}
\label{sec:methForm}

We consider the initial value problem
\begin{equation}
\label{initValueProb}
M u'(t) = F\big(t,u(t)\big),\qquad u(t_0) = u_0 \in \RR^d,
\end{equation}
where $M \in \RR^{d\times d}$ is a regular matrix and $F$, sufficiently smooth,
satisfies a Lipschitz-condition with respect to the second variable.
Furthermore, let $I=(t_0,t_0+T]$ be an arbitrary but fixed time interval with
positive length $T$. The value $u_0$ at $t=t_0$ will be called the initial
value in the following.

If the ODE system~\eqref{initValueProb} originates from a finite element
semi-discretization in space of a parabolic partial differential equation
then $M$ is the time-constant mass matrix. The explicit appearance of $M$
allows to see easily if for certain ideas some systems of linear equations
have to be solved and additional effort is necessary.

To describe the vector-valued case ($d>1$) in an easy
way, let $(\cdot,\cdot)$ be the standard inner product and $\|\cdot\|$ the
Euclidean norm on $\RR^d$, $d\in\NN$. Besides, let $e_j$ be the $j$th standard unit
vector in $\RR^d$, $1 \leq j \leq d$.

For an arbitrary interval $J$ and $q\in\NN$, the spaces of continuous and
$p$ times continuously differentiable $\RR^q$-valued functions on $J$ are
written as $C(J,\RR^q)$ and $C^p(J,\RR^q)$, respectively. Furthermore, the
space of square-integrable $\RR^q$-valued functions shall be denoted by
$L^2(J,\RR^q)$ or, for convenience, sometimes also by $C^{-1}(J,\RR^q)$.
For $s \in \ZZ$, $s \geq 0$, we write $P_s(J, \RR^q)$ for the space of
$\RR^q$-valued polynomials on $J$ of degree less than or equal to $s$. Further notation
will be introduced later at the beginning of the sections where it is
needed.

In order to describe the methods, we need some time mesh. Therefore,
the interval $I$ is decomposed by
\begin{gather*}
t_0 < t_1 < \dots < t_{N-1} < t_N = t_0+T
\end{gather*}
into $N$ disjoint subintervals $I_n := (t_{n-1},t_n]$, $n = 1,\ldots, N$.
Furthermore, we set
\begin{gather*}
\tau_n := t_n - t_{n-1}, \qquad
\tau := \max_{1\le n\le N} \tau_n.
\end{gather*}
For any piecewise continuous function $v$ we define by
\begin{gather*}
v(t_n^+) := \lim_{t\to t_n+0} v(t),\qquad
v(t_n^-) := \lim_{t\to t_n-0} v(t),\qquad
[v]_n := v(t_n^+) - v(t_n^-)
\end{gather*}
the one-sided limits and the jump of $v$ at $t_n$.

In this paper $C$ denotes a generic constant independent of the time mesh
parameter $\tau$.

\subsection{Local formulation}

Let $r, k \in \ZZ$ with $0 \leq k \leq r$. In order to numerically solve
the initial value problem~\eqref{initValueProb}, we shall introduce special
variational time discretization methods $\vtd{r}{k}$ with parameters $r$ and $k$.

Using a standard time-marching strategy, the discrete solution is successively
determined on $I_n$, $n=1,\ldots,N$, by local problems of the form \medskip

Find $U\in P_r(I_n,\RR^d)$ such that
\begin{subequations}
\label{eq:locProb}
\begin{align}
U(t_{n-1}^+)
& = U(t_{n-1}^-),
&& \text{if } k \geq 1, 
\label{eq:locProbI} \\
M U^{(i+1)}(t_n^-)
& = \frac{\d^i}{\d t^i} \Big(F\big(t,U(t)\big)\Big)\Big|_{t=t_n^-},
&& \text{if } k \geq 2, \, i = 0, \ldots, \left\lfloor\tfrac{k}{2}\right\rfloor - 1,
\label{eq:locProbE} \\
M U^{(i+1)}(t_{n-1}^+)
& = \frac{\d^i}{\d t^i} \Big(F\big(t,U(t)\big)\Big)\Big|_{t=t_{n-1}^+},
&& \text{if } k \geq 3, \, i = 0, \ldots, \left\lfloor\tfrac{k-1}{2}\right\rfloor - 1,
\label{eq:locProbA}
\end{align}
and
\begin{multline}
\label{eq:locProbVar}
\IIn{\big( M U',\varphi\big)}
+ \delta_{0,k} \big(M \big[U\big]_{n-1}, \varphi(t_{n-1}^+)\big)
= \IIn{\big( F(\cdot,U(\cdot)), \varphi\big)}
\quad \forall \varphi\in P_{r-k}(I_n,\RR^d)
\end{multline}
\end{subequations}
where $U(t_0^-) = u_0$ and $\delta_{i,j}$ denotes the Kronecker Delta. Moreover, the integrator $\II$
represents either the integral over $I_n$ or the application of a quadrature formula for
approximate integration. Details will be described later on.

Note that the formulation can be easily extended to the case $k = r+1$.
Then the variational condition~\eqref{eq:locProbVar} needs to hold formally
for all $\varphi \in P_{-1}(I_n,\RR^d)$ that should be interpreted as
``there is no variational condition''. Hence, only conditions at both ends
of the interval $I_n$ are used.

The method $\vtd{r}{k}$ can be shortly described by
\begin{alignat*}{4}
& \text{trial space: } && P_r, &\qquad& \text{if } k \geq 1:  \text{initial condition}, &&\\
& \text{test space: }  && P_{r-k}, && \text{if } k \geq 2: \text{$\mathrm{ODE}^{(i)}$ in $t_n^-$}, &i & = 0, \ldots,
\left\lfloor \tfrac{k}{2}\right\rfloor - 1,\\
&  && && \text{if } k \geq 3: \text{$\mathrm{ODE}^{(i)}$ in $t_{n-1}^+$}, & i & = 0, \ldots, \left\lfloor
\tfrac{k-1}{2}\right\rfloor - 1.
\end{alignat*}
The notation $\mathrm{ODE}^{(i)}$ means that the discrete solution
fulfills the $i$th derivative of the system of ordinary differential
equations. Counting the number of conditions leads for $k\ge 1$ to
\begin{gather*}
\dim P_{r-k} + 1 + \left\lfloor \tfrac{k}{2} \right\rfloor + \left\lfloor \tfrac{k-1}{2} \right\rfloor
= r-k + 1 + 1 + \tfrac{k}{2} + \tfrac{k-1}{2} - \tfrac{1}{2}
= r-k+2+k-1 = r+1
\end{gather*}
while we have $\dim P_{r} = r+1$ conditions if $k = 0$. The number of
degrees of freedom equals for all $k$ to $\dim P_r = r+1$. Hence, the
number of conditions coincides for all cases with the number of degrees of
freedom.

\begin{remark}
\label{rem:dGcGP}
The $\vtd{r}{k}$ framework generalizes two well-known types of variational time
discretization methods. The method $\vtd{r}{0}$ is the discontinuous Galerkin method
$\mathrm{dG}(r)$ whereas the method $\vtd{r}{1}$ equates to the continuous
Galerkin--Petrov method $\mathrm{cGP}(r)$.

On closer considerations we see that methods $\vtd{r}{k}$ with even $k$
are $\mathrm{dG}$-like since there are point conditions on the 
$\left\lfloor \frac{k}{2} \right\rfloor$th derivative of the discrete
solution but this derivative might be discontinuous. The methods $\vtd{r}{k}$
with odd $k$ are $\mathrm{cGP}$-like since there are point conditions up to
the $\left\lfloor \frac{k}{2} \right\rfloor$th derivative of the discrete
solution and this derivative is continuous. We have in detail
\begin{gather*}
\vtd{r}{k} \mathrel{\widehat{=}}
\begin{cases}
\mathrm{dG}(r),
& k=0, \\
\mathrm{cGP}(r),
& k = 1, \\
\mathrm{dG}\text{-}C^{\left\lfloor \frac{k-1}{2} \right\rfloor}(r),
& k \geq 2, \, k \text{ even}, \\
\mathrm{cGP}\text{-}C^{\left\lfloor \frac{k-1}{2} \right\rfloor}(r),
& k \geq 3, \, k \text{ odd},
\end{cases}
\end{gather*}
where we used and generalized the definitions and notation of~\cite{MS11}.
Note that there is also another reason for naming the methods like this. All methods with odd $k$
share their A-stability with the cGP method while methods with even $k$ are
strongly A-stable as the dG method.
\end{remark}

In order to obtain a fully computable discrete problem usually a quadrature formula
$\Qun$ is chosen as integrator, i.e., $\II = \Qun$.
To indicate this choice, we simply write $\Qun$-$\vtd{r}{k}$. Moreover, we agree
that integration over $I_n$ is used if no quadrature rule is specified. We shall
mostly use quadrature rules that are exact for polynomials of degree up to $2r-k$.
This ensures in the case of an affine linear right-hand side $F(t,u) = f(t)-Au$ with
time-independent $A$ that all $u$ depending terms in~\eqref{eq:locProbVar} are integrated
exactly.

The special structure of the method~\eqref{eq:locProb} motivates to use an assigned
interpolation operator that conserves derivatives at the end points of the interval
up to a certain order. In detail, we define on the reference interval $[-1,1]$
the interpolation operator
$\Ihat{r}{k} : C^{\left\lfloor \frac{k}{2} \right\rfloor}([-1,1]) \to P_r([-1,1])$
that uses the interpolation points
\begin{equation}
\label{quadrule:Qpp}
\begin{alignedat}{2}
& \text{at the left end: }
&& \text{derivatives up to order $\left\lfloor\tfrac{k-1}{2}\right\rfloor$ in $-1^+$,} \\
& \text{at the right end: }
&& \text{derivatives up to order $\left\lfloor\tfrac{k}{2}\right\rfloor$ in $1^-$,} \\
& \text{in the interior: }
&& \text{zeros $\hat{t}_i \in (-1,1)$ of the $(r-k)$th Jacobi-polynomial} \\
&&& \text{with respect to the weight
$(1+\hat{t})^{\left\lfloor \frac{k-1}{2}\right\rfloor +1} (1-\hat{t})^{\left\lfloor \frac{k}{2}\right\rfloor +1}$}.
\end{alignedat}
\end{equation}
Note that there is not even a single point value at the left end for $k=0$. Thinking of
multiple counting in the case of derivatives, a total number of
\begin{gather*}
r-k + \left\lfloor \tfrac{k}{2} \right\rfloor + 1 + \left\lfloor \tfrac{k-1}{2} \right\rfloor + 1
= r - k + k -1 + 2 = r + 1
\end{gather*}
interpolation points is obtained. Hence, the number of conditions coincides with the
dimension of $P_r$. The interpolation operator $\Ihat{r}{k}$ is of Hermite-type and
provides the standard error estimates for Hermite interpolation.

In addition, we define by
\begin{equation*}
\Qhat{r}{k}[\hat{f}]:=\int_{-1}^1 (\Ihat{r}{k}\hat{f})(\hat{t})\,\d\hat{t}
\end{equation*}
a quadrature rule on $[-1,1]$ that is in a natural way assigned to the method
$\vtd{r}{k}$. The quadrature rules $\Qhat{r}{k}$ are known in the literature as
generalized Gauss--Radau or Gauss--Lobatto formulas, respectively, see
e.g.~\cite{Gau04,JB09}. The weights of the quadrature rule $\Qhat{r}{k}$ could
be calculated by integrating the appropriate Hermite basis functions on $[-1,1]$.
Finally, we obtain
\begin{gather*}
\int_{-1}^{1} \widehat{\varphi}(\hat{t}) \,\d \hat{t}
\approx \Qhat{r}{k}\big[\widehat{\varphi}\big]
= \int_{-1}^{1} (\Ihat{r}{k}\widehat{\varphi})(\hat{t}) \,\d \hat{t}
= \sum_{i=0}^{\left\lfloor \frac{k-1}{2}\right\rfloor} w_i^L \widehat{\varphi}^{(i)}(-1^+)
+ \sum_{i=1}^{r-k} w_i^I \widehat{\varphi}(\hat{t}_{i})
+ \sum_{i=0}^{\left\lfloor \frac{k}{2}\right\rfloor} w_i^R \widehat{\varphi}^{(i)}(+1^-).
\end{gather*}
The quadrature rule $\Qhat{r}{k}$ is exact for polynomials up to degree $2r-k$. It can
be shown that all quadrature weights are different from zero, see~\cite{JB09}.
In detail, we have
\begin{gather*}
w_j^I > 0, \qquad w_j^L > 0, \qquad (-1)^j w_j^R > 0,
\end{gather*}
so even the sign of the weights is known. Note that in general
(for $k \geq 2$) not all weights are positive. Semi-explicit or recursive
formulas for the weights of these methods can be found in~\cite{Pet17}.

Transferring the quadrature rule $\Qhat{r}{k}$ and the interpolation operator $\Ihat{r}{k}$
from $[-1,1]$ to the interval $I_n$, we obtain $\Q{r}{k}$ and $\I{r}{k}$ where we usually skip
to indicate the interval $I_n$ since this will be clear from context. Hence, we have
\begin{gather*}
\int_{I_n} \varphi(t) \,\d t
\approx \Q{r}{k}\big[\varphi\big]
= \frac{\tau_n}{2} \Bigg[
\sum_{i=0}^{\left\lfloor \frac{k-1}{2}\right\rfloor} w_i^L \left(\tfrac{\tau_n}{2}\right)^{\!i} \varphi^{(i)}(t_{n-1}^+)
+ \sum_{i=1}^{r-k} w_i^I \varphi(t_{n,i})
+ \sum_{i=0}^{\left\lfloor \frac{k}{2}\right\rfloor} w_i^R \left(\tfrac{\tau_n}{2}\right)^{\!i} \varphi^{(i)}(t_n^-)
\Bigg]
\end{gather*}
where $t_{n,i} = \frac{1}{2}\left(t_{n-1}+t_n+\tau_n \hat{t}_i\right) \in I_n$,
$i =1,\ldots,r-k$. The appearing factor
$\left(\tfrac{\tau_n}{2}\right)^{\!i}$ originates from the chain rule.

\begin{remark}
The quadrature rule $\Q{r}{0}$ is the well-known right-sided Gauss--Radau quadrature
formula with $r+1$ points that is typically used for the discontinuous Galerkin method $\mathrm{dG}(r)$.
$\Q{r}{1}$ is the Gauss--Lobatto quadrature rule with $r+1$ points that is often used together with
the continuous Galerkin--Petrov method $\mathrm{cGP}(r)$.
\end{remark}

\begin{remark}
For $1 \leq k \leq r$ the $\vtd{r}{k}$ method with exact integration could be
analyzed also in a generalization of the unified framework of~\cite{AMN11} as we
shall show below, see~\eqref{eq:locProbUnified}. Note that the dG-method ($k=0$)
has been already fitted in this framework there.

We define for $k\ge 1$ a projection operator
$\Pint_n : C^{\left\lfloor \frac{k}{2}
\right\rfloor-1}(\overline{I}_n,\RR^d) \to P_{r-1}(\overline{I}_n,\RR^d)$ by
\begin{subequations}
\label{eq:Pint}
\begin{align}
(\Pint_n v)^{(i)}(t_n^-)
& = v^{(i)}(t_n^-),
&& \text{if } k \geq 2, \, i = 0, \ldots, \left\lfloor\tfrac{k}{2}\right\rfloor - 1,
\label{eq:PintE} \\
(\Pint_n v)^{(i)}(t_{n-1}^+)
& = v^{(i)}(t_{n-1}^+),
&& \text{if } k \geq 3, \, i = 0, \ldots, \left\lfloor\tfrac{k-1}{2}\right\rfloor - 1,
\label{eq:PintA}
\end{align}
and
\begin{equation}
\label{eq:PintVar}
\int_{I_n} \big( \Pint_n v(t),\varphi(t)\big)\,\d t = 
\int_{I_n} \big( v(t), \varphi(t)\big)\,\d t
\qquad\forall \varphi\in P_{r-k}(I_n,\RR^d).
\end{equation}
\end{subequations}
Then an equivalent formulation of~\eqref{eq:locProb} with $1 \leq k \leq r$ 
and exact integration reads
\medskip

Find $U\in P_r(I_n,\RR^d)$ with given $U(t_{n-1})\in\RR^d$ such that
\begin{equation}
\label{eq:locProbUnified}
M U'(t) = \Pint_n F\big(t,U(t)\big) \qquad \forall t \in I_n
\end{equation}
where $U(t_0) = u_0$.

Indeed, if $U$ solves~\eqref{eq:locProb} then $U' \in P_{r-1}(I_n,\RR^d)$ obviously
satisfies all conditions of~\eqref{eq:Pint} with $v= F\big(\cdot,U(\cdot)\big)$.
Since $\Pint_n v$ is uniquely defined we directly get~\eqref{eq:locProbUnified}.

Otherwise let $U$ solve~\eqref{eq:locProbUnified}. Since
there are polynomials on both sides, we can differentiate the equation
by any order. With~\eqref{eq:PintE} and~\eqref{eq:PintA} we have
\begin{gather*}
M U^{(i+1)}(\tilde{t})
= \frac{\d^i}{\d t^i} \Big(\Pint_n F\big(t,U(t)\big)\Big) \Big|_{t=\tilde{t}}
= \frac{\d^i}{\d t^i} \Big( F\big(t,U(t)\big)\Big) \Big|_{t=\tilde{t}}
\end{gather*}
for $\tilde{t}=t_n^-$ and all $i = 0, \ldots, \left\lfloor \frac{k}{2}\right\rfloor -1$ if $k\geq 2$,
as well as for $\tilde{t}=t_{n-1}^+$ and all $i = 0, \ldots, \left\lfloor \frac{k-1}{2}\right\rfloor -1$
if $k\geq 3$, respectively. Hence, the
conditions~\eqref{eq:locProbE} and~\eqref{eq:locProbA} hold.
Taking the inner product of~\eqref{eq:locProbUnified} with an arbitrary $\varphi \in P_{r-k}(I_n,\RR^d)$
and integrating over $I_n$ yield together with~\eqref{eq:PintVar}
\begin{gather*}
\int_{I_n} \Big( M U'(t),\varphi(t)\Big)\,\d t
= \int_{I_n} \Big( \Pint_n F\big(t,U(t)\big), \varphi(t)\Big)\,\d t
= \int_{I_n} \Big( F\big(t,U(t)\big), \varphi(t)\Big)\,\d t
\end{gather*}
which is~\eqref{eq:locProbVar} with exact integration.

Note that certain numerically integrated versions of $\vtd{r}{k}$ could also be written in
the form~\eqref{eq:locProbUnified} if appropriate projections are applied.
\end{remark}

\subsection{Global formulation}

For $s \in \ZZ$, $s \geq 0$, we define the space $\Y{s}$ of vector-valued piecewise
polynomials of maximal degree $s$ by
\begin{gather*}
\Y{s} := \left\{ \varphi\in L^2(I,\RR^d)\::\: \varphi|_{I_n}
\in P_s(I_n, \RR^d),\, n=1,\dots, N\right\}.
\end{gather*}

Studying the conditions~\eqref{eq:locProbI}, \eqref{eq:locProbE}, and~\eqref{eq:locProbA}
we see that the solution $U$ of $\II$-$\vtd{r}{k}$ is $\left\lfloor \frac{k-1}{2} \right\rfloor$
times continuously differentiable on $I$ if $F$ is sufficiently smooth. Furthermore, the condition~\eqref{eq:locProbE}
for $U \in C^{\left\lfloor \frac{k-1}{2} \right\rfloor}(I)$ already implies~\eqref{eq:locProbA} for $n \geq 2$.
Consequently, the method could be reformulated as follows \medskip

Find $U \in \Y{r} \cap C^{\left\lfloor \frac{k-1}{2} \right\rfloor}(I,\RR^d)$ such that
\begin{subequations}
\label{def:globalDisc}
\begin{align}
U^{(i)}(t_0^+)
& = U^{(i)}(t_0^-),
&& \text{if } k \geq 1, \, i = 0,\ldots, \left\lfloor \tfrac{k-1}{2} \right\rfloor, 
\label{def:globalDiscI}\\
M U^{(i+1)}(t_n^-)
& = \frac{\d^i}{\d t^i} \Big(F\big(t,U(t)\big)\Big)\Big|_{t=t_n^-},
&& \text{if } k \geq 2, \, i = 0, \ldots, \left\lfloor\tfrac{k}{2}\right\rfloor - 1,
\label{def:globalDiscE}
\end{align}
for all $n = 1, \ldots, N$, and
\begin{equation}
\label{def:globalDiscVar}
\sum_{n=1}^N \left\{\IIn{\big(M U' - F(\cdot,U(\cdot)), \varphi\big)}
+ \delta_{0,k} \big(M \big[U\big]_{n-1}, \varphi(t_{n-1}^+)\big) \right\}
= 0 \qquad \forall \varphi \in \Y{r-k}
\end{equation}
\end{subequations}
where $U^{(i)}(t_0^-) = u^{(i)}(t_0)$, $0 \leq i \leq \left\lfloor \frac{k}{2} \right\rfloor$,
which includes the initial value $u_0$ in the problem formulation. We agree
on defining $u^{(j)}(t_0)$ recursively using the differential equation, i.e.,
\begin{equation}
\label{eq:IC}
\begin{alignedat}{2}
u^{(0)}(t_0) & := u_0,&\qquad
M u^{(2)}(t_0) & := \partial_t F\big(t_0, u(t_0)\big) + \partial_u F\big(t_0,u(t_0)\big) u^{(1)}(t_0), \\
M u^{(1)}(t_0) & := F\big(t_0,u(t_0)\big), &\qquad
M u^{(j)}(t_0) & := \frac{\d^{j-1}}{\d t^{j-1}} F\big(t,u(t)\big) \big|_{t=t_0}, \, j \geq 3.
\end{alignedat}
\end{equation}
The term $\frac{\d^{j-1}}{\d t^{j-1}} F\big(t,u(t)\big) \big|_{t=t_0}$ depends only
on $u(t_0), \ldots, u^{(j-1)}(t_0)$ and can be calculated using some generalization of
Fa\`{a} di Bruno's formula, see e.g.~\cite{EM03, M00}. If $F$ is affine linear in $u$, i.e.,
$F(t,u(t)) = f(t)-A(t)u(t)$, then we simply have
\begin{align*}
M u^{(j)}(t_0)
& := \frac{\d^{j-1}}{\d t^{j-1}} F\big(t,u(t)\big)\big|_{t=t_0}
   = f^{(j-1)}(t_0) - \sum_{l=0}^{j-1} \tbinom{j-1}{l} A^{(j-1-l)}(t_0) u^{(l)}(t_0), \quad j \geq 1,
\end{align*}
by Leibniz' rule for the $(j-1)$th derivative.

Note that since the test space $\Y{r-k}$ in~\eqref{def:globalDiscVar} allows
discontinuities at the boundaries of
subintervals, the problem can be decoupled by choosing test functions $\varphi$
supported on a single time interval $I_n$ only. Moreover, exploiting for
$k\ge 1$ the property $U \in C^{\left\lfloor \frac{k-1}{2} \right\rfloor}$
as well as~\eqref{def:globalDiscI} and~\eqref{def:globalDiscE}, we also
obtain~\eqref{eq:locProbI} and~\eqref{eq:locProbA}.
Therefore, the global problem~\eqref{def:globalDisc} can be converted back
into a sequence of local problems~\eqref{eq:locProb} in time on the different
subintervals $I_n$, $n=1,\dots, N$.

On each subinterval $I_n$, the local solution $U|_{I_n}$ belongs to
$P_r(I_n,\RR^d)$. Hence, it is completely described by $r+1$ vector
coefficients from $\RR^d$. However,
$\left\lfloor \frac{k-1}{2} \right\rfloor+1=
\left\lfloor \frac{k+1}{2} \right\rfloor$ of them are already fixed
by data inherited from the previous subinterval $I_{n-1}$ or the initial
conditions since $U \in C^{\left\lfloor \frac{k-1}{2} \right\rfloor}$. This
means that the size of the system to be solved on each subinterval is $r-
\left\lfloor \frac{k-1}{2} \right\rfloor$. The extreme case $k=r$ leads to
a system size of $\left\lfloor \frac{r}{2} \right\rfloor+1$ that is roughly
half of the size $r+1$ obtained for $k=0$.

\section{Postprocessing}
\label{sec:postprocessing}
We shall present a simple postprocessing in this section.

Recall~\eqref{quadrule:Qpp} for the definition of the quadrature points of the
quadrature rule $\Q{r}{k}$ which is exact for polynomials up to degree $2r-k$.

\begin{theorem}[Postprocessing $\Q{r}{k}$-$\vtd{r}{k}$ $\leadsto$ $\Q{r}{k}$-$\vtd{r+1}{k+2}$]
\label{th:postproc}
Let $r, k \in \ZZ$, $0 \leq k \leq r$, and suppose that $U \in \Y{r}$ solves
$\Q{r}{k}$-$\vtd{r}{k}$. For every $n = 1, \ldots, N$ set
\begin{gather*}
\widetilde{U} \big|_{I_n} = U\big|_{I_n} + a_n \vartheta_n,
\qquad \vartheta_n \in P_{r+1}(I_n,\RR),
\end{gather*}
where $\vartheta_n$ vanishes in all $(r+1)$ quadrature points of $\Q{r}{k}$
and additionally satisfies $\vartheta_n^{\left(\left\lfloor \frac{k}{2} \right\rfloor +1\right)}(t_n^-) = 1$
while the vector $a_n \in \RR^d$ is defined by
\begin{gather}\label{eq:an}
a_n = M^{-1} \left( \frac{\d^{\left\lfloor \frac{k}{2} \right\rfloor}}{\d t^{\left\lfloor \frac{k}{2} \right\rfloor}}
F\big(t,U(t)\big) \Big|_{t=t_n^-} - M U^{\left(\left\lfloor \frac{k}{2} \right\rfloor+1\right)}(t_n^-)\right)\!.
\end{gather}
Moreover, let $\widetilde{U}(t_0^-) = U(t_0^-)$. Then $\widetilde{U} \in \Y{r+1}$ solves
$\Q{r}{k}$-$\vtd{r+1}{k+2}$.
\end{theorem}
\begin{proof}
We have to verify that $\widetilde{U}$ satisfies all conditions for $\Q{r}{k}$-$\vtd{r+1}{k+2}$
where $\Q{r}{k}$ is the quadrature rule associated to $\vtd{r}{k}$ which is
exact for polynomials up to degree $2r-k$.

First of all we show an identity needed later. The special form of $\vartheta_n$,
the exactness of $\Q{r}{k}$, and integration by parts yield
\begin{gather}
\label{help:quadZeroPP}
\begin{aligned}
\Q{r}{k}\big[\vartheta_n' \varphi\big]
  & = \int_{I_n} \!\vartheta_n'(t) \varphi(t)\,\d t
= - \int_{I_n} \!\vartheta_n(t) \varphi'(t)\,\d t + (\vartheta_n \varphi)\big|_{t_{n-1}^+}^{t_n^-} \\
  & = - \underbrace{\Q{r}{k}\big[\vartheta_n \varphi'\big]}_{=0} - \delta_{0,k} (\vartheta_n \varphi)(t_{n-1}^+)
= - \delta_{0,k} (\vartheta_n \varphi)(t_{n-1}^+) \qquad \forall \varphi \in P_{r-k}(I_n,\RR).
\end{aligned}
\end{gather}
Precisely, we used that both $\vartheta_n' \varphi$ and $\vartheta_n \varphi'$ are polynomials
of maximal degree $2r-k$ and that $\vartheta_n$ vanishes in all quadrature points,
especially in $t_n^-$ and for $k\ge 1$ also in $t_{n-1}^+$.

For $k \geq 1$ we have $\vartheta_n(t_{n-1}^+) = \vartheta_n(t_n^-) = 0$.
Therefore, the initial condition holds due to
$\widetilde{U}(t_{n-1}^+) = U(t_{n-1}^+) = U(t_{n-1}^-) = \widetilde{U}(t_{n-1}^-)$.
For $k=0$ it is somewhat more complicated to prove $\widetilde{U}(t_{n-1}^+) = \widetilde{U}(t_{n-1}^-)$,
for details see~\eqref{enu:3} below. The remaining conditions can be verified as follows.
\begin{enumerate}[(i)]
\item \label{enu:1} \underline{Conditions at $t_n^-$ for
$0 \leq i \leq \left\lfloor\tfrac{k+2}{2}\right\rfloor-2 = \left\lfloor\tfrac{k}{2}\right\rfloor -1$:} \\
We obtain from the definitions of $\widetilde{U}$ and $U$
\begin{align*}
M \widetilde{U}^{(i+1)}(t_n^-)
= M U^{(i+1)}(t_n^-) + M a_n \underbrace{\vartheta_n^{(i+1)}(t_n^-)}_{=0} 
= \frac{\d^i}{\d t^i} F\big(t,U(t)\big) \Big|_{t=t_n^-} 
= \frac{\d^i}{\d t^i} F\big(t,\widetilde{U}(t)\big) \Big|_{t=t_n^-}
\end{align*}
since the derivatives of $U$ and $\widetilde{U}$ in $t_n^-$ coincide up to order
$\left\lfloor \tfrac{k}{2} \right\rfloor$ due to the definition of $\vartheta_n$.
\item \label{enu:2} \underline{Condition at $t_n^-$ for $i = \left\lfloor\tfrac{k+2}{2}\right\rfloor-1 = \left\lfloor\tfrac{k}{2}\right\rfloor$:} \\
Just like above we get, additionally using the definition of $a_n$, 
\begin{align*}
M \widetilde{U}^{\left(\left\lfloor \frac{k}{2} \right\rfloor+1\right)}(t_n^-)
& = M U^{\left(\left\lfloor \frac{k}{2} \right\rfloor+1\right)}(t_n^-)
+ M a_n \underbrace{\vartheta_n^{\left(\left\lfloor \frac{k}{2} \right\rfloor+1\right)}(t_n^-)}_{=1} \\
& = M U^{\left(\left\lfloor \frac{k}{2} \right\rfloor+1\right)}(t_n^-)
+ \frac{\d^{\left\lfloor \frac{k}{2} \right\rfloor}}{\d t^{\left\lfloor \frac{k}{2} \right\rfloor}} F\big(t,U(t)\big) \Big|_{t=t_n^-}
- M U^{\left(\left\lfloor \frac{k}{2} \right\rfloor+1\right)}(t_n^-) \\
& = \frac{\d^{\left\lfloor \frac{k}{2} \right\rfloor}}{\d t^{\left\lfloor \frac{k}{2} \right\rfloor}} F\big(t,U(t)\big) \Big|_{t=t_n^-}
  = \frac{\d^{\left\lfloor \frac{k}{2} \right\rfloor}}{\d t^{\left\lfloor \frac{k}{2} \right\rfloor}} F\big(t,\widetilde{U}(t)\big) \Big|_{t=t_n^-}.
\end{align*}
\item \label{enu:3} \underline{Variational condition:} \\
We have to prove that $\Q{r}{k}\big[(M \widetilde{U}',\varphi)\big] = \Q{r}{k}\big[(F(\cdot,\widetilde{U}(\cdot)),\varphi)\big]$
for all $\varphi \in P_{(r+1)-(k+2)}(I_n,\RR^d)$. Actually, we can even test with functions $\varphi \in P_{r-k}(I_n,\RR^d)$.

We first study the case $k \geq 1$. By the definitions of $\widetilde{U}$ and $U$,
the identity~\eqref{help:quadZeroPP}, and the fact that $U$ and $\widetilde{U}$ coincide
at all quadrature points we have
\begin{align*}
\Q{r}{k}\Big[\big(M \widetilde{U}',\varphi\big)\Big]
& = \Q{r}{k}\Big[\big(M U',\varphi\big)\Big] + \Q{r}{k}\Big[\big(M a_n \vartheta_n',\varphi\big)\Big] \\
& = \Q{r}{k}\Big[\big(F(\cdot,U(\cdot)),\varphi\big)\Big]
+ \underbrace{\Q{r}{k}\Big[\vartheta_n' \big(M a_n,\varphi\big)\Big]}_{
\substack{=0, \text{ since} \\ (M a_n,\varphi) \in P_{r-k}(I_n,\RR)}} \\
& = \Q{r}{k}\Big[\big(F(\cdot,\widetilde{U}(\cdot)),\varphi\big)\Big]
\qquad \qquad \forall \varphi \in P_{r-k}(I_n,\RR^d).
\end{align*}

Now let $k=0$. The same arguments as for $k \geq 1$ yield for all $\varphi \in P_{r}(I_n,\RR^d)$
\begin{gather*}
\Q{r}{0}\Big[\big(M \widetilde{U}',\varphi\big)\Big]
= \Q{r}{0}\Big[\big(F(\cdot,\widetilde{U}(\cdot)),\varphi\big)\Big]
- \big(M \big[U\big]_{n-1}, \varphi(t_{n-1}^+)\big)
- \vartheta_n(t_{n-1}^+) \big(M a_n,\varphi(t_{n-1}^+)\big).
\end{gather*}
We study the last two terms. Using the definitions of the jump $\big[U\big]_{n-1}$ and of $\widetilde{U}$,
we find
\begin{gather}
\label{help:anK0PP}
\big[U\big]_{n-1} + a_n \vartheta_n(t_{n-1}^+)
= \widetilde{U}(t_{n-1}^+) - U(t_{n-1}^-) = \big[\widetilde{U}\big]_{n-1}
\end{gather}
where we also exploited that $\vartheta_{n-1}(t_{n-1}^-)=0$. Hence, we have
\begin{gather}
\label{help:varCondPP}
\Q{r}{0}\Big[\big(M \widetilde{U}',\varphi\big)\Big] + \big(M \big[\widetilde{U}\big]_{n-1}, \varphi(t_{n-1}^+)\big)
= \Q{r}{0}\Big[\big(F(\cdot,\widetilde{U}(\cdot)),\varphi\big)\Big]
\qquad \forall \varphi \in P_{r}(I_n,\RR^d).
\end{gather}
Choosing the special test functions $\varphi_j \in P_{r}(I_n,\RR^d)$, $j=1,\dots,d$, that
vanish in the $r$ inner quadrature points of $\Q{r}{0}$ and satisfy
$\varphi_j(t_{n-1}^+) = e_j$ as well as having in
mind~\eqref{enu:2}, we component-wise
find $\big[\widetilde{U}\big]_{n-1} = 0$. Thereby, at once we have proven the initial
condition and verified the needed variational condition since now also the
jump term in~\eqref{help:varCondPP} can be dropped.
\item \label{enu:4}
\underline{Conditions at $t_{n-1}^+$ for
$0 \leq i \leq \left\lfloor\tfrac{k+2-1}{2}\right\rfloor-2 = \left\lfloor\tfrac{k-1}{2}\right\rfloor -1$:} \\
With an argumentation similar to that in~\eqref{enu:1} we gain
\begin{align*}
M \widetilde{U}^{(i+1)}(t_{n-1}^+)
& = M U^{(i+1)}(t_{n-1}^+) + M a_n \underbrace{\vartheta_n^{(i+1)}(t_{n-1}^+)}_{=0} \\
& = \frac{\d^i}{\d t^i} F\big(t,U(t)\big) \Big|_{t=t_{n-1}^+}
= \frac{\d^i}{\d t^i} F\big(t,\widetilde{U}(t)\big) \Big|_{t=t_{n-1}^+}\!.
\end{align*}
\item \label{enu:5}
\underline{Condition at $t_{n-1}^+$ for $i = \left\lfloor\tfrac{k+2-1}{2}\right\rfloor-1 = \left\lfloor\tfrac{k-1}{2}\right\rfloor$ if $k \geq 1$:} \\
It remains to prove that
\begin{gather*}
M \widetilde{U}^{\left(\left\lfloor\frac{k-1}{2}\right\rfloor+1\right)}(t_{n-1}^+)
= \frac{\d^{\left\lfloor \frac{k-1}{2} \right\rfloor}}{\d t^{\left\lfloor \frac{k-1}{2} \right\rfloor}}
F\big(t,\widetilde{U}(t)\big) \Big|_{t=t_{n-1}^+}.
\end{gather*}
We use the variational condition for $\widetilde{U}$ with specially
chosen test functions $\varphi_j \in P_{r-k}(I_n,\RR^d)$, $j=1,\dots,d$,
that vanish at all inner quadrature points of $\Q{r}{k}$, i.e.,
\begin{gather*}
\varphi_j(t_{n,i}) = 0,\quad i = 1,\ldots,r-k,
\qquad \qquad \text{and satisfy} \qquad \qquad
\varphi_j(t_{n-1}^+) = e_j.
\end{gather*}
As shown in~\eqref{enu:3} we have
\begin{gather*}
\Q{r}{k}\Big[\big(M \widetilde{U}',\varphi_j\big)\Big]
= \Q{r}{k}\Big[\big(F(\cdot,\widetilde{U}(\cdot)),\varphi_j\big)\Big],
\qquad j=1,\dots,d,
\end{gather*}
since $k\ge 1$.
The special choices of $\varphi_j$, the definition of the quadrature rule, and the already known
identities from~\eqref{enu:1}, \eqref{enu:2}, and~\eqref{enu:4} yield after a short calculation
using Leibniz' rule for the $i$th derivative that
\begin{align*}
&& \Q{r}{k}\Big[\big(M \widetilde{U}',\varphi_j\big)\Big]
& = \Q{r}{k}\Big[\big(F(\cdot,\widetilde{U}(\cdot)),\varphi_j\big)\Big],
                                        \quad j=1,\dots,d, \\
\Leftrightarrow
&& w_{\left\lfloor\frac{k-1}{2}\right\rfloor}^L
M \widetilde{U}^{\left(\left\lfloor\frac{k-1}{2}\right\rfloor+1\right)}
                                                (t_{n-1}^+) \cdot \underbrace{\varphi_j(t_{n-1}^+)}_{=e_j}
& =\begin{multlined}[t]
w_{\left\lfloor\frac{k-1}{2}\right\rfloor}^L
\frac{\d^{\left\lfloor \frac{k-1}{2} \right\rfloor}}{\d t^{\left\lfloor \frac{k-1}{2} \right\rfloor}} F\big(t,\widetilde{U}(t)\big)
\Big|_{t=t_{n-1}^+} \!\!\!\cdot \underbrace{\varphi_j(t_{n-1}^+)}_{=e_j}, \\
j=1,\dots,d,
\end{multlined} \\
\Leftrightarrow
&& M \widetilde{U}^{\left(\left\lfloor\frac{k-1}{2}\right\rfloor+1\right)}(t_{n-1}^+)
& =\frac{\d^{\left\lfloor \frac{k-1}{2} \right\rfloor}}{\d t^{\left\lfloor \frac{k-1}{2} \right\rfloor}} F\big(t,\widetilde{U}(t)\big)
\Big|_{t=t_{n-1}^+}.
\end{align*}
Note that we also used that $w_{\left\lfloor\frac{k-1}{2}\right\rfloor}^L \neq 0$.
\end{enumerate}
Collecting the above arguments, we see that $\widetilde{U}$ solves $\Q{r}{k}$-$\vtd{r+1}{k+2}$.
\end{proof}

From the definition~\eqref{eq:an}, it seems that a linear system with the mass matrix $M$
has to be solved in every time step in order to obtain the correction vector $a_n$.
However, the computational costs for calculating $a_n$ can be reduced significantly
as we shall show now.

\begin{proposition}
\label{prop:altern_an}
The correction vectors $a_n \in \RR^d$ defined in~\eqref{eq:an} for the
postprocessing presented in Theorem~\ref{th:postproc} can be alternatively
calculated by
\begin{gather*}
a_n = \frac{-1}{\vartheta_n^{\left(\left\lfloor \frac{k-1}{2} \right\rfloor +1\right)}(t_{n-1}^+)}
\left(U^{\left(\left\lfloor \frac{k-1}{2} \right\rfloor +1\right)}(t_{n-1}^+)
- \widetilde{U}^{\left(\left\lfloor \frac{k-1}{2} \right\rfloor +1\right)}
                                                (t_{n-1}^-)\right)
\qquad \text{for $n > 1$},
\end{gather*}
and
\begin{gather*}
a_1 = \frac{-1}{\vartheta_1^{\left(\left\lfloor \frac{k-1}{2} \right\rfloor +1\right)}(t_0^+)}
\left(U^{\left(\left\lfloor \frac{k-1}{2} \right\rfloor +1\right)}(t_0^+)
- u^{\left(\left\lfloor \frac{k-1}{2} \right\rfloor +1\right)}(t_0)\right)
\end{gather*}
where $u^{\left(\left\lfloor \frac{k-1}{2} \right\rfloor +1\right)}(t_0)$ is defined in~\eqref{eq:IC}.
\end{proposition}
\begin{proof}
For $k=0$, we get from~\eqref{help:anK0PP} combined with 
$\big[\widetilde{U}\big]_{n-1}=0$, which was shown just below~\eqref{help:varCondPP}, that
$a_n = \frac{-1}{\vartheta_n(t_{n-1}^+)} \big[U\big]_{n-1}
= \frac{-1}{\vartheta_n(t_{n-1}^+)} \big(U(t_{n-1}^+)-\widetilde{U}(t_{n-1}^-)\big)$.
Taking $\widetilde{U}(t_0^-) = U(t_0^-) = u(t_0) = u_0$ into account, we are done in this case.

Otherwise, for $k \geq 1$, using the definition of the postprocessing and~\eqref{enu:5}
of the proof of Theorem~\ref{th:postproc}, we obtain that
\begin{gather*}
M U^{\left(\left\lfloor \frac{k-1}{2} \right\rfloor +1\right)}(t_{n-1}^+)
+ M a_n \vartheta_n^{\left(\left\lfloor \frac{k-1}{2} \right\rfloor +1\right)}(t_{n-1}^+)
= M \widetilde{U}^{\left(\left\lfloor\frac{k-1}{2}\right\rfloor+1\right)}(t_{n-1}^+)
= \frac{\d^{\left\lfloor \frac{k-1}{2} \right\rfloor}}{\d t^{\left\lfloor \frac{k-1}{2} \right\rfloor}}
F\big(t,\widetilde{U}(t)\big) \Big|_{t=t_{n-1}^+}\!.
\end{gather*}
Furthermore, we have $\vartheta_n^{(i)}(t_{n-1}^+) = 0$ for $i = 0, \ldots, \left\lfloor \frac{k-1}{2} \right\rfloor$
and therefore
\begin{gather}
\label{eq:f_UtildeToU}
\frac{\d^{\left\lfloor \frac{k-1}{2} \right\rfloor}}{\d t^{\left\lfloor \frac{k-1}{2} \right\rfloor}} F\big(t,\widetilde{U}(t)\big) \Big|_{t=t_{n-1}^+}
= \frac{\d^{\left\lfloor \frac{k-1}{2} \right\rfloor}}{\d t^{\left\lfloor \frac{k-1}{2} \right\rfloor}} F\big(t,U(t)\big) \Big|_{t=t_{n-1}^+}.
\end{gather}
Since $F$ is sufficiently smooth and $U$ is $\left\lfloor \frac{k-1}{2} \right\rfloor$ times continuously differentiable we get
\begin{align*}
\frac{\d^{\left\lfloor \frac{k-1}{2} \right\rfloor}}{\d t^{\left\lfloor \frac{k-1}{2} \right\rfloor}} F\big(t,U(t)\big) \Big|_{t=t_{n-1}^+}
& \!\!= \frac{\d^{\left\lfloor \frac{k-1}{2} \right\rfloor}}{\d t^{\left\lfloor \frac{k-1}{2} \right\rfloor}}
F\big(t,U(t)\big) \Big|_{t=t_{n-1}^-} \\
& \!\!= \frac{\d^{\left\lfloor \frac{k-1}{2} \right\rfloor}}{\d t^{\left\lfloor \frac{k-1}{2} \right\rfloor}}
F\big(t,\widetilde{U}(t)\big) \Big|_{t=t_{n-1}^-}
\!\!= M \widetilde{U}^{(\left\lfloor \frac{k-1}{2} \right\rfloor +1)}(t_{n-1}^-),
\qquad n > 1,
\end{align*}
where also $\vartheta_{n-1}^{(i)}(t_{n-1}^-) = 0$ for $i = 0, \ldots, \left\lfloor \frac{k}{2} \right\rfloor$
and~\eqref{enu:1} or~\eqref{enu:2} of the proof of Theorem~\ref{th:postproc} were used.
Altogether exploiting that $M$ is regular an easy manipulation of the identities yields
\begin{gather*}
a_n = \frac{-1}{\vartheta_n^{\left(\left\lfloor \frac{k-1}{2} \right\rfloor +1\right)}(t_{n-1}^+)}
\left(U^{\left(\left\lfloor \frac{k-1}{2} \right\rfloor +1\right)}(t_{n-1}^+)
- \widetilde{U}^{\left(\left\lfloor \frac{k-1}{2} \right\rfloor +1\right)}(t_{n-1}^-)\right)\!,
\qquad n > 1.
\end{gather*}

A similar formula can also be established for $n=1$. Since $U$ satisfies~\eqref{eq:locProbA}
and $U(t_0) := u_0$, we obviously obtain, recalling the definition~\eqref{eq:IC} of $u^{(i)}(t_0)$, that
\begin{gather*}
U^{(i)}(t_0^+) = u^{(i)}(t_0)
\qquad \text{for $i = 0,\ldots,\left\lfloor \tfrac{k-1}{2} \right\rfloor$.}
\end{gather*}
Therefore, we have in~\eqref{eq:f_UtildeToU} for $n=1$
\begin{gather*}
\frac{\d^{\left\lfloor \frac{k-1}{2} \right\rfloor}}{\d t^{\left\lfloor \frac{k-1}{2} \right\rfloor}} F\big(t,\widetilde{U}(t)\big) \Big|_{t=t_0^+}
= \frac{\d^{\left\lfloor \frac{k-1}{2} \right\rfloor}}{\d t^{\left\lfloor \frac{k-1}{2} \right\rfloor}} F\big(t,U(t)\big) \Big|_{t=t_0^+}
= \frac{\d^{\left\lfloor \frac{k-1}{2} \right\rfloor}}{\d t^{\left\lfloor \frac{k-1}{2} \right\rfloor}} F\big(t,u(t)\big) \Big|_{t=t_0^+}
= M u^{\left(\left\lfloor \frac{k-1}{2} \right\rfloor+1\right)}(t_0).
\end{gather*}
This results in
\begin{gather*}
a_1 = \frac{-1}{\vartheta_1^{\left(\left\lfloor \frac{k-1}{2} \right\rfloor +1\right)}(t_0^+)}
\left(U^{\left(\left\lfloor \frac{k-1}{2} \right\rfloor +1\right)}(t_0^+)
- u^{\left(\left\lfloor \frac{k-1}{2} \right\rfloor +1\right)}(t_0)\right)\!.
\end{gather*}
Hence, the alternative calculation provides the same correction vector.
\end{proof}

Note that $a_n$ can be calculated in this way without solving a system of linear equations.
From the structure of $a_n$ we see that the postprocessing can be interpreted
as a correction of the jump in the lowest order derivative of the discrete solution that
is not continuous by construction.

Since the division by $\vartheta_n^{\left(\left\lfloor \frac{k-1}{2} \right\rfloor +1\right)}(t_{n-1}^+)$
changes the normalization of $\vartheta_n$ only, we conclude the following.

\begin{corollary}[Alternative postprocessing $\Q{r}{k}$-$\vtd{r}{k}$ $\leadsto$ $\Q{r}{k}$-$\vtd{r+1}{k+2}$]
\label{cor:postprocAlt}
Let $r, k \in \ZZ$, $0 \leq k \leq r$, and suppose that $U \in \Y{r}$ solves
$\Q{r}{k}$-$\vtd{r}{k}$. For every $n = 1, \ldots, N$ set
\begin{gather*}
\widetilde{U} \big|_{I_n} = U\big|_{I_n} - \tilde{a}_n \tilde{\vartheta}_n,
\qquad \tilde{\vartheta}_n \in P_{r+1}(I_n,\RR),
\end{gather*}
where $\tilde{\vartheta}_n(t) = \vartheta_n(t)/\vartheta_n^{\left(\left\lfloor \frac{k-1}{2}\right\rfloor +1\right)}(t_{n-1}^+)$
with $\vartheta_n$ from Theorem~\ref{th:postproc}, i.e., $\tilde{\vartheta}_n$ vanishes in all $(r+1)$ quadrature
points of $\Q{r}{k}$ and additionally satisfies $\tilde{\vartheta}_n^{\left(\left\lfloor \frac{k-1}{2} \right\rfloor +1\right)}(t_{n-1}^+) = 1$.
The vector $\tilde{a}_n \in \RR^d$ is defined by
\begin{gather*}
\tilde{a}_n
:= \begin{cases}
U^{\left(\left\lfloor \frac{k-1}{2} \right\rfloor +1\right)}(t_0^+)
- u^{\left(\left\lfloor \frac{k-1}{2} \right\rfloor +1\right)}(t_0), & n=1, \\
U^{\left(\left\lfloor \frac{k-1}{2} \right\rfloor +1\right)}(t_{n-1}^+)
- \widetilde{U}^{\left(\left\lfloor \frac{k-1}{2} \right\rfloor +1\right)}(t_{n-1}^-),
                                          & n>1,
\end{cases}
\end{gather*}
where $u^{\left(\left\lfloor \frac{k-1}{2} \right\rfloor +1\right)}(t_0)$ is given by~\eqref{eq:IC}.
Moreover, let $\widetilde{U}(t_0^-) = U(t_0^-)$. Then $\widetilde{U} \in \Y{r+1}$ solves
$\Q{r}{k}$-$\vtd{r+1}{k+2}$.
\end{corollary}

A direct proof for the alternative postprocessing is given in Appendix~\ref{app:postAlt}.

\section[Connections to collocation methods]{Connections between numerically integrated variational time discretization methods and collocation methods}
\label{sec:collocation}

We shall prove that the (local) solution of $\Q{r}{k}$-$\vtd{r+1}{l}$ with
$1 \leq l \leq k+2$ (which obviously includes $\Q{r}{k}$-$\vtd{r+1}{k+2}$)
can be characterized as the solution of the (local) collocation problem
with respect to the quadrature points of $\Q{r}{k}$, i.e.,

\medskip

Find $\widetilde{U} \in P_{r+1}(I_n,\RR^d)$ with given $\widetilde{U}(t_{n-1}) \in \RR^d$ such that
\begin{subequations}
\label{eq:collocation}
\begin{align}
M \widetilde{U}^{(i+1)}(t_n^-)
& = \frac{\d^i}{\d t^i} \Big( F\big(t,\widetilde{U}(t)\big) \Big)\Big|_{t=t_n^-},
&& i = 0, \ldots, \left\lfloor\tfrac{k}{2}\right\rfloor\!,
\label{eq:collocationE}\\
M \widetilde{U}^{(i+1)}(t_{n-1}^+)
& = \frac{\d^i}{\d t^i} \Big( F\big(t,\widetilde{U}(t)\big) \Big)\Big|_{t=t_{n-1}^+},
&& \text{if } k \geq 1, \, i = 0, \ldots, \left\lfloor\tfrac{k-1}{2}\right\rfloor\!,
\label{eq:collocationA}\\
M \widetilde{U}'(t_{n,i})
& = F\big(t_{n,i},\widetilde{U}(t_{n,i})\big),
&& i = 1, \ldots, r-k,
\label{eq:collocationI}
\end{align}
\end{subequations}
where $\widetilde{U}(t_0) = u_0$. Here, $t_{n,i}$ are the zeros of the $(r-k)$th
Jacobi-polynomial with respect to the weight
$(1+\hat{t})^{\left\lfloor \frac{k-1}{2}\right\rfloor +1} (1-\hat{t})^{\left\lfloor \frac{k}{2}\right\rfloor +1}$
transformed to the interval $[t_{n-1},t_n]$, see also~\eqref{quadrule:Qpp}.

Methods similar to~\eqref{eq:collocation} are known as
collocation methods with multiple nodes, see e.g.~\cite[p.~275]{HNW08}.
Unfortunately, existing results in the literature often neglect to study the unique
solvability or conditions on $F$ are not explicitly given. However, the connections mentioned
above directly imply that all these methods are equivalent as we will prove now.

\begin{theorem}[Equivalence to collocation methods]
\label{th:equiCollocation}
Let $r, k, l \in \ZZ$, $0 \leq k \leq r$, and $1 \leq l \leq k+2$. Then
$\widetilde{U} \in P_{r+1}(I_n,\RR^d)$ solves $\Q{r}{k}$-$\vtd{r+1}{l}$ if and
only if $\widetilde{U}$ solves the collocation method~\eqref{eq:collocation}
with respect to the quadrature points of $\Q{r}{k}$.
\end{theorem}
\begin{proof}
For clarity and convenience, we recall the conditions of the
$\Q{r}{k}$-$\vtd{r+1}{l}$ method with $0 \leq k \leq r$ and
$1 \leq l \leq k+2$.
\medskip

Given $\widetilde{U}(t_{n-1}^-)$, find $\widetilde{U} \in P_{r+1}(I_n,\RR^d)$ such that
\begin{subequations}
\label{eq:equi_l}
\begin{align}
\widetilde{U}(t_{n-1}^+)
& = \widetilde{U}(t_{n-1}^-),
&& \nonumber \\
M \widetilde{U}^{(i+1)}(t_{n}^-)
& = \frac{\d^i}{\d t^i} \Big(F\big(t,\widetilde{U}(t)\big)\Big)\Big|_{t=t_{n}^-},
&& \text{if } l \geq 2, \, i = 0,\ldots,\left\lfloor \tfrac{l}{2} \right\rfloor - 1, \label{eq:equiE_l} \\
M \widetilde{U}^{(i+1)}(t_{n-1}^+)
& = \frac{\d^i}{\d t^i} \Big(F\big(t,\widetilde{U}(t)\big)\Big)\Big|_{t=t_{n-1}^+},
&& \text{if } l \geq 3, \, i = 0,\ldots,\left\lfloor \tfrac{l-1}{2} \right\rfloor - 1, \label{eq:equiA_l} \\
\Q{r}{k}\Big[\big(M\widetilde{U}',\varphi\big)\Big]
& = \Q{r}{k}\Big[\big(F(\cdot,\widetilde{U}(\cdot)),\varphi\big)\Big]
&& \forall \varphi \in P_{r+1-l}(I_n,\RR^d). \label{eq:equiVar_l}
\end{align}
\end{subequations}

First of all, assume that $\widetilde{U}$ solves~\eqref{eq:collocation}. Then because
of~\eqref{eq:collocationE} and~\eqref{eq:collocationA} obviously $\widetilde{U}$
satisfies the conditions~\eqref{eq:equiE_l} and~\eqref{eq:equiA_l} since $l\le k+2$.
In order to gain a better understanding of the numerically integrated variational
condition, we have a look at its detailed definition. For
$\varphi \in C^{\left\lfloor \frac{k}{2} \right\rfloor}(\overline{I}_n,\RR^d)$ we have
by definition of the quadrature rule
\begin{align*}
\Q{r}{k}\Big[\big(M\widetilde{U}'-F(\cdot,\widetilde{U}(\cdot)),\varphi\big)\Big] = \,
& \frac{\tau_n}{2}
  \Bigg[ \sum_{i=0}^{\left\lfloor \frac{k-1}{2} \right\rfloor} \!\!w_i^L \left(\tfrac{\tau_n}{2}\right)^{\!i}
\frac{\d^i}{\d t^i} \big(M \widetilde{U}'(t) - F(t,\widetilde{U}(t)),\varphi(t)\big) \Big|_{t=t_{n-1}^+} \\
& \hspace{3em} + \sum_{i=1}^{r-k} \!w_i^I \big(M \widetilde{U}'(t_{n,i})-F(t_{n,i},\widetilde{U}(t_{n,i})),\varphi(t_{n,i})\big) \\
& \hspace{3em} + \sum_{i=0}^{\left\lfloor \frac{k}{2} \right\rfloor} \!w_i^R \left(\tfrac{\tau_n}{2}\right)^{\!i}
\frac{\d^i}{\d t^i} \big(M \widetilde{U}'(t) - F(t,\widetilde{U}(t)),\varphi(t)\big) \Big|_{t=t_{n}^-} \Bigg].
\end{align*}
Applying Leibniz' rule for the $i$th derivative, the right-hand side above can be rewritten as
\begin{align*}
& \frac{\tau_n}{2}
  \Bigg[ \sum_{i=0}^{\left\lfloor \frac{k-1}{2} \right\rfloor} \!\!w_i^L \left(\tfrac{\tau_n}{2}\right)^{\!i}
\sum_{j=0}^i \tbinom{i}{j} \big(M \widetilde{U}^{(j+1)}(t_{n-1}^+)
- \tfrac{\d^j}{\d t^j} F(t,\widetilde{U}(t)) \big|_{t=t_{n-1}^+},\varphi^{(i-j)}(t_{n-1}^+)\big) \\
& \hspace{3em} + \sum_{i=1}^{r-k} \!w_i^I \big(M \widetilde{U}'(t_{n,i}) - F(t_{n,i},\widetilde{U}(t_{n,i})),\varphi(t_{n,i})\big) \\
& \hspace{3em} + \sum_{i=0}^{\left\lfloor \frac{k}{2} \right\rfloor} \!w_i^R \left(\tfrac{\tau_n}{2}\right)^{\!i}
\sum_{j=0}^i \tbinom{i}{j} \big(M \widetilde{U}^{(j+1)}(t_{n}^-)
- \tfrac{\d^j}{\d t^j} F(t,\widetilde{U}(t)) \big|_{t=t_{n}^-},\varphi^{(i-j)}(t_{n}^-)\big) \Bigg].
\end{align*}
Using the collocation conditions~\eqref{eq:collocationA}, \eqref{eq:collocationI},
and~\eqref{eq:collocationE}, we see that all three sums equal to $0$.
Hence, we obtain
\begin{gather}
\label{eq:collocationVar}
\Q{r}{k}\Big[\big(M\widetilde{U}'-F(\cdot,\widetilde{U}(\cdot)),\varphi\big)\Big]
= 0 \qquad \forall \varphi \in C^{\left\lfloor \frac{k}{2} \right\rfloor}(\overline{I}_n,\RR^d)
\end{gather}
which immediately gives~\eqref{eq:equiVar_l}.

Now, we study the other direction and assume that $\widetilde{U}$ solves~\eqref{eq:equi_l}. In order
to verify~\eqref{eq:collocationE}, for example, we need to prove that~\eqref{eq:equiE_l}
also holds for $i = \left\lfloor \tfrac{l}{2} \right\rfloor, \ldots, \left\lfloor \tfrac{k}{2} \right\rfloor$.
For this purpose (in the case that $\left\lfloor \tfrac{l}{2} \right\rfloor \leq \left\lfloor \tfrac{k}{2} \right\rfloor$)
we use the special test function
\begin{gather*}
\psi = c \left(t-t_{n-1}\right)^{\textstyle\left\lfloor \tfrac{k-1}{2} \right\rfloor - \left\lfloor \tfrac{l-1}{2} \right\rfloor +1}
\left(t-t_n\right)^{\textstyle\left\lfloor \tfrac{k}{2} \right\rfloor - \left\lfloor \tfrac{l}{2} \right\rfloor}
\prod_{i=1}^{r-k} \left(t-t_{n,i}\right) \in P_{r+1-l}(I_n,\RR^d)
\end{gather*}
with an arbitrary vector $c \in \RR^d$ where $t_{n,i}$, $i=1,\ldots, r-k$, denote the
inner quadrature points of $\Q{r}{k}$. Then~\eqref{eq:equiVar_l}, the special
construction of $\psi$, \eqref{eq:equiE_l}, and~\eqref{eq:equiA_l} yield
\begin{align*}
0 & = \Q{r}{k}\Big[\big(M\widetilde{U}' - F(\cdot,\widetilde{U}(\cdot)),\psi\big)\Big] \\
  & = \sum_{i=0}^{\textstyle\left\lfloor \frac{k-1}{2} \right\rfloor} w_i^L \left(\tfrac{\tau_n}{2}\right)^{i+1}
\sum_{j=0}^i \tbinom{i}{j} \Big(M\widetilde{U}^{(j+1)}(t_{n-1}^+)
- \tfrac{\d^j}{\d t^j} F\big(t,\widetilde{U}(t)\big)\big|_{t=t_{n-1}^+},\psi^{(i-j)}(t_{n-1}^+)\Big) \\
  & \qquad + \sum_{i=0}^{\textstyle\left\lfloor \frac{k}{2} \right\rfloor} w_i^R \left(\tfrac{\tau_n}{2}\right)^{i+1}
\sum_{j=0}^i \tbinom{i}{j} \Big(M\widetilde{U}^{(j+1)}(t_{n}^-)
- \tfrac{\d^j}{\d t^j} F\big(t,\widetilde{U}(t)\big)\big|_{t=t_{n}^-},\psi^{(i-j)}(t_{n}^-)\Big) \\
  & = \sum_{i=0}^{\textstyle\left\lfloor \frac{k-1}{2} \right\rfloor} w_i^L \left(\tfrac{\tau_n}{2}\right)^{i+1}
\sum_{j=\textstyle\left\lfloor \tfrac{l-1}{2} \right\rfloor}^i \tbinom{i}{j} \Big(M\widetilde{U}^{(j+1)}(t_{n-1}^+)
- \tfrac{\d^j}{\d t^j} F\big(t,\widetilde{U}(t)\big)\big|_{t=t_{n-1}^+},\psi^{(i-j)}(t_{n-1}^+)\Big) \\
  & \qquad + \sum_{i=0}^{\textstyle\left\lfloor \frac{k}{2} \right\rfloor} w_i^R \left(\tfrac{\tau_n}{2}\right)^{i+1}
\sum_{j=\textstyle\left\lfloor \tfrac{l}{2} \right\rfloor}^i \tbinom{i}{j} \Big(M\widetilde{U}^{(j+1)}(t_{n}^-)
- \tfrac{\d^j}{\d t^j} F\big(t,\widetilde{U}(t)\big)\big|_{t=t_{n}^-},\psi^{(i-j)}(t_{n}^-)\Big).
\end{align*}
Furthermore, we have that $\psi^{(i)}(t_{n-1}^+) = 0$ for
$i=0,\ldots, \left\lfloor \tfrac{k-1}{2} \right\rfloor - \left\lfloor \tfrac{l-1}{2}\right\rfloor$
and $\psi^{(i)}(t_{n}^-) = 0$ for $i=0,\ldots, \left\lfloor \tfrac{k}{2} \right\rfloor - \left\lfloor \tfrac{l}{2} \right\rfloor - 1$.
Thus, the above identity simplifies to
\begin{gather*}
0 = w_i^R \left(\tfrac{\tau_n}{2}\right)^{i+1} \tbinom{i}{j} \Big(M\widetilde{U}^{(j+1)}(t_{n}^-)
- \tfrac{\d^j}{\d t^j} F\big(t,\widetilde{U}(t)\big)\big|_{t=t_{n}^-},\psi^{(i-j)}(t_{n}^-)\Big)
\end{gather*}
with $i= \left\lfloor \frac{k}{2} \right\rfloor$ and $j = \left\lfloor \frac{l}{2} \right\rfloor$.
Since $w_i^R \neq 0$, $\psi^{\left(\left\lfloor \frac{k}{2} \right\rfloor-\left\lfloor \frac{l}{2} \right\rfloor\right)}(t_{n}^-) \neq 0$, 
and the vector $c \in \RR^d$ can be chosen arbitrarily, it follows
\begin{gather*}
M \widetilde{U}^{(i+1)}(t_{n}^-) = \frac{\d^i}{\d t^i} \Big(F\big(t,\widetilde{U}(t)\big)\Big)\Big|_{t=t_{n}^-}
\qquad \text{for } i = \left\lfloor \tfrac{l}{2} \right\rfloor\!.
\end{gather*}
Using the test functions
\begin{gather*}
\psi = c \left(t-t_{n-1}\right)^{\textstyle\left\lfloor \tfrac{k-1}{2} \right\rfloor - \left\lfloor \tfrac{l-1}{2} \right\rfloor +1}
\left(t-t_n\right)^{\textstyle\left\lfloor \tfrac{k}{2} \right\rfloor - \left\lfloor \tfrac{l}{2} \right\rfloor - j}
\prod_{i=1}^{r-k} \left(t-t_{n,i}\right) \in P_{r+1-l-j}(I_n,\RR^d)
\end{gather*}
with an arbitrary vector $c \in \RR^d$ and
$j = 1, \ldots, \left\lfloor \tfrac{k}{2} \right\rfloor - \left\lfloor \tfrac{l}{2} \right\rfloor$
we iteratively also prove
\begin{gather*}
M \widetilde{U}^{(i+1)}(t_{n}^-) = \frac{\d^i}{\d t^i} \Big(F\big(t,\widetilde{U}(t)\big)\Big)\Big|_{t=t_{n}^-}
\qquad \text{for } i = \left\lfloor \tfrac{l}{2} \right\rfloor + 1, \ldots, \left\lfloor \tfrac{k}{2} \right\rfloor\!.
\end{gather*}
A similar argument can also be used for the missing point conditions at $t_{n-1}^+$.
Note that we needed for the above implications that the weights of $\Q{r}{k}$ do not vanish which has been
proven in~\cite{JB09}.

It remains to verify~\eqref{eq:collocationI}. Since we already know that $\widetilde{U}$ satisfies
the collocation conditions~\eqref{eq:collocationE} and~\eqref{eq:collocationA}, the
variational condition~\eqref{eq:equiVar_l} reduces to
\begin{gather}
\label{eq:connection1_helpIdentity}
\sum_{i=1}^{r-k} \!w_i^I \big(M \widetilde{U}'(t_{n,i}),\varphi(t_{n,i})\big)
= \sum_{i=1}^{r-k} \!w_i^I \big(F(t_{n,i},\widetilde{U}(t_{n,i})),\varphi(t_{n,i})\big)
\qquad \forall \varphi \in P_{r+1-l}(I_n, \RR^d).
\end{gather}
Also recall that $w_j^I > 0$ for all $1 \leq j \leq r-k$. Thus, choosing
in~\eqref{eq:connection1_helpIdentity} test functions of the form
\begin{gather*}
\varphi_{j}(t) = c \prod_{\substack{i=1 \\ i\neq j}}^{r-k} (t-t_{n,i})
\in P_{r-k-1}(I_n, \RR^d) \subseteq P_{r+1-l}(I_n, \RR^d)
\end{gather*}
with an arbitrary vector $c \in \RR^d$, we get the collocation
condition~\eqref{eq:collocationI} in $t_{n,j}$.

Hence, the stated equivalence has been proven.
\end{proof}

Summarizing, we have shown that a solution of $\Q{r}{k}$-$\vtd{r+1}{k+2}$
also solves $\Q{r}{k}$-$\vtd{r+1}{l}$ with $1 \leq l \leq k+2$ as well as a
collocation with respect to the quadrature points of $\Q{r}{k}$ and vice
versa. Shortly, we have
\begin{align*}
\Q{r}{k}\text{-}\vtd{r+1}{k+2}
& \quad \mathrel{\widehat{=}} \quad
  \Q{r}{k}\text{-}\vtd{r+1}{l} \quad \text{with} \quad 1 \leq l \leq k+2 \\
& \quad \mathrel{\widehat{=}} \quad
  \text{collocation with respect to the quadrature points of } \Q{r}{k}.
\end{align*}

\begin{remark}
Independent of the above findings, the connection between collocation
methods and (postprocessed) numerically integrated discontinuous Galerkin
methods (using the right-sided Gauss--Radau quadrature), i.e.,
Theorem~\ref{th:equiCollocation} for $k = 0 \leq r$ and $l=2$, was already
observed in~\cite{VR15}. Moreover, connections between collocation methods
and the numerically integrated continuous Galerkin--Petrov methods (using
interpolatory quadrature formulas with as many quadrature points as number
of independent variational conditions) are shown in~\cite{Hul72b,Hul72a}.
\end{remark}

\subsection{Error estimates for collocation methods}

The method defined in~\eqref{eq:collocation} is a collocation method with
multiple nodes as considered for example in~\cite[p.~275]{HNW08}. Thus, the
usual error analysis for collocation methods also applies here,
provided that $F$ and $u$ are sufficiently smooth. According
to~\cite[p.~276]{HNW08}, we have the following error estimate.

\begin{proposition}
\label{prop:collErrorTimePoints}
The collocation method~\eqref{eq:collocation} possesses the same order $2r-k+1$ as the
underlying quadrature formula $\Q{r}{k}$. Hence, it holds
\begin{gather*}
\max_{1 \leq n \leq N} \big\|(u-\widetilde{U})(t_n^-)\big\|
\leq C(F,u) \tau^{2r-k+1}
\end{gather*}
where $\widetilde{U}$ and $u$ denote the solutions of the collocation method~\eqref{eq:collocation}
and the initial value problem~\eqref{initValueProb}, respectively.
\end{proposition}
Moreover, global error estimates can be shown by adapting techniques presented
in~\cite[Theorem~2]{Hul72b}. Together with~\cite[p.~276, pp.~212--214]{HNW08}, we obtain the
following.

\begin{proposition}
\label{prop:collErrorGlobal}
Let $\widetilde{U}$ denote the solution of the collocation method~\eqref{eq:collocation}
and $u$ the exact solution of~\eqref{initValueProb}.
Then we have
\begin{align*}
\sup_{t \in I} \big\|(u-\widetilde{U})(t)\big\|
& \leq C(F,u) \tau^{\min\{2r-k+1,(r+1)+1\}}
\intertext{and}
\sup_{t\in I} \big\|(u-\widetilde{U})^{(l)}(t)\big\|
& \leq C(F,u) \tau^{(r+1)+1-l},\qquad 1 \leq l \leq r+1,
\end{align*}
where the derivatives are understood in an interval-wise sense.
\end{proposition}
The term $2r-k+1$ inside the minimum is due to the fact that the
convergence order of these collocation methods is limited by the accuracy
of the underlying quadrature formula $\Q{r}{k}$ that is exactly $2r-k+1$.
Note that the limitation is active for $r=k$ only.

\subsection{Reversed postprocessing}

In Section~\ref{sec:postprocessing} we studied a postprocessing for $\Q{r}{k}$-$\vtd{r}{k}$.
We have seen that starting from a solution $U$ of $\Q{r}{k}$-$\vtd{r}{k}$
we can easily construct a solution $\widetilde{U}$ of
$\Q{r}{k}$-$\vtd{r+1}{k+2}$. This already implies uniqueness of $U$
provided that solutions of $\Q{r}{k}$-$\vtd{r+1}{k+2}$
or~\eqref{eq:collocation} are unique. Indeed, if $U^1$ and $U^2$ solve
$\Q{r}{k}$-$\vtd{r}{k}$ and their postprocessed solutions are identical,
then by construction of the postprocessing $U^1$ and $U^2$ coincide in the
$(r+1)$ quadrature points of $\Q{r}{k}$. Thus, since both are polynomials
of degree $r$, it follows $U^1 \equiv U^2$.

We now ask whether or not this postprocessing step can be reversed for $0 \leq k \leq r$.
Then this would imply the solvability of $\Q{r}{k}$-$\vtd{r}{k}$ provided that
$\Q{r}{k}$-$\vtd{r+1}{k+2}$ or~\eqref{eq:collocation} has a solution.

\begin{proposition}[Reversed postprocessing]
Let $r, k \in \ZZ$, $0 \leq k \leq r$, and suppose that $\widetilde{U} \in \Y{r+1}$ solves
$\Q{r}{k}$-$\vtd{r+1}{k+2}$. Then $\I{r}{k} \widetilde{U} \in \Y{r}$ solves
$\Q{r}{k}$-$\vtd{r}{k}$.
\end{proposition}
\begin{proof}
Let $\widetilde{U} \in P_{r+1}(I_n,\RR^d)$ solve $\Q{r}{k}$-$\vtd{r+1}{k+2}$. We
shall prove that then $\I{r}{k} \widetilde{U} \in P_{r}(I_n,\RR^d)$ is a solution
of $\Q{r}{k}$-$\vtd{r}{k}$. Since $\I{r}{k}$ conserves the derivatives up to
order $\left\lfloor \frac{k}{2} \right\rfloor$ at $t_{n}^-$ and up to order
$\left\lfloor \frac{k-1}{2} \right\rfloor$ at $t_{n-1}^+$, respectively, we have
\begin{align*}
M \big(\I{r}{k} \widetilde{U}\big)^{(i+1)}(t_n^-)
= M \widetilde{U}^{(i+1)}(t_n^-)
& = \frac{\d^i}{\d t^i} F\big(t,\widetilde{U}(t)\big) \Big|_{t=t_n^-}\\
& = \frac{\d^i}{\d t^i} F\big(t,\I{r}{k}\widetilde{U}(t)\big) \Big|_{t=t_n^-},
\qquad 
\text{if } k \geq 2, \, i = 0, \ldots, \left\lfloor\tfrac{k}{2}\right\rfloor-1, &
\end{align*}
and analogously
\begin{gather*}
M \big(\I{r}{k} \widetilde{U}\big)^{(i+1)}(t_{n-1}^+)
= \frac{\d^i}{\d t^i} F\big(t,\I{r}{k}\widetilde{U}(t)\big) \Big|_{t=t_{n-1}^+},
\qquad \text{if } k \geq 3, \, i = 0, \ldots, \left\lfloor\tfrac{k-1}{2}\right\rfloor-1.
\end{gather*}
These are in fact~\eqref{eq:locProbE} and~\eqref{eq:locProbA}.

It remains to prove that $\I{r}{k} \widetilde{U}$ also satisfies the variational
condition~\eqref{eq:locProbVar} with $\II = \Q{r}{k}$.
According to~\eqref{eq:collocationVar} we have for $\widetilde{U}$ that
\begin{gather*}
\Q{r}{k}\Big[\big(M\widetilde{U}',\varphi\big)\Big]
= \Q{r}{k}\Big[\big(F(\cdot,\widetilde{U}(\cdot)),\varphi\big)\Big]
\qquad \forall \varphi \in P_{r-k}(I_n,\RR^d).
\end{gather*}
Note that originally from the definition of $\Q{r}{k}$-$\vtd{r+1}{k+2}$ the
variational condition is postulated for $\varphi \in P_{r-k-1}(I_n,\RR^d)$ only.
Since $\I{r}{k} \widetilde{U} - \widetilde{U} \in P_{r+1}(I_n,\RR^d)$ vanishes in
the quadrature points of $\Q{r}{k}$ it holds
\begin{gather*}
\I{r}{k} \widetilde{U} - \widetilde{U} = c_n \tilde{\vartheta}_n
\qquad \text{with} \qquad
c_n = \big(\I{r}{k} \widetilde{U} - \widetilde{U}\big)^{(\left\lfloor \frac{k-1}{2} \right\rfloor +1)}
                        (t_{n-1}^+)\in\RR^d
\end{gather*}
and $\tilde{\vartheta}_n \in P_{r+1}(I_n,\RR)$ as defined in Corollary~\ref{cor:postprocAlt}.
Hence, using~\eqref{help:quadZeroPP} we conclude
\begin{align*}
\Q{r}{k}\Big[\big(M (\I{r}{k} \widetilde{U})',\varphi\big)\Big]
& = \Q{r}{k}\Big[\big(M \widetilde{U}',\varphi\big)\Big]
+ \Q{r}{k}\Big[\big(M c_n \tilde{\vartheta}_n',\varphi\big)\Big] \\
& = \Q{r}{k}\Big[\big(F(\cdot,\widetilde{U}(\cdot)),\varphi\big)\Big]
+ \Q{r}{k}\Big[\tilde{\vartheta}_n' \big(M c_n,\varphi\big)\Big] \\
& = \Q{r}{k}\Big[\big(F(\cdot,\I{r}{k} \widetilde{U}(\cdot)),\varphi\big)\Big]
- \delta_{0,k} \tilde{\vartheta}_n(t_{n-1}^+) \big(M c_n,\varphi(t_{n-1}^+)\big)
\quad \forall \varphi \in P_{r-k}(I_n,\RR^d)
\end{align*}
which, because of
\begin{gather*}
\tilde{\vartheta}_n(t_{n-1}^+) M c_n
= M \big(\I{r}{0} \widetilde{U} - \widetilde{U}\big)(t_{n-1}^+)
= M \big(\I{r}{0} \widetilde{U}(t_{n-1}^+) - \widetilde{U}(t_{n-1}^-)\big)
= M \big[\I{r}{0} \widetilde{U}\big]_{n-1}
\end{gather*}
for $k=0$, completes the argument.
\end{proof}

\subsection{Consequences for existence, uniqueness, and error estimates}

The connections between numerically integrated variational time
discretization methods and collocation methods with multiple nodes observed
in the previous subsections can now be used to obtain results on the
existence and uniqueness of solutions of $\Q{r}{k}$-$\vtd{r}{k}$ as well as
give rise to global error estimates and superconvergence estimates in the
time mesh points.

\begin{corollary}[Existence and uniqueness]
If there is a solution $\widetilde{U} \in P_{r+1}(I_n,\RR^d)$ of the
collocation method with multiple nodes defined by~\eqref{eq:collocation}
then $U = \I{r}{k} \widetilde{U} \in P_{r}(I_n,\RR^d)$ solves
$\Q{r}{k}$-$\vtd{r}{k}$. Furthermore, if $\widetilde{U}$ is uniquely
defined as solution of~\eqref{eq:collocation} then so is $U$ as solution of
$\Q{r}{k}$-$\vtd{r}{k}$.
\end{corollary}

\begin{corollary}[Global error estimates]
\label{cor:globalErrorEst}
Let the error estimates of Proposition~\ref{prop:collErrorGlobal}
hold for the solution $\widetilde{U}$ of~\eqref{eq:collocation} and the exact
solution $u$ of~\eqref{initValueProb}. Then we have for the solution
$U$ of $\Q{r}{k}$-$\vtd{r}{k}$ that
\begin{align*}
\sup_{t \in I} \big\|(u-U)(t)\big\|
& \leq C(F,u) \tau^{r+1}
\intertext{and}
\sup_{t \in I} \big\|(u-U)^{(l)}(t)\big\|
& \leq C(F,u) \tau^{r+1-l},
\qquad 1 \leq l \leq r,
\end{align*}
where the derivatives are understood in an interval-wise sense.
\end{corollary}
\begin{proof}
Since $U = \I{r}{k} \widetilde{U}$, Proposition~\ref{prop:collErrorGlobal}
and standard error estimates yield
\begin{align*}
\sup_{t \in I} \big\|(u-U)(t)\big\|
& \leq \sup_{t \in I} \big\|(u-\widetilde{U})(t)\big\|
+ \sup_{t \in I} \big\|(\widetilde{U} - \I{r}{k} \widetilde{U})(t)\big\| \\
& \leq C(F,u) \tau^{\min\{2r-k+1,(r+1)+1\}} + C \tau^{r+1} \sup_{t \in I} \|\widetilde{U}^{(r+1)}(t)\|.
\end{align*}
In order to estimate $\sup_{t \in I} \|\widetilde{U}^{(r+1)}(t)\|$ we again use
Proposition~\ref{prop:collErrorGlobal} as follows
\begin{gather*}
\sup_{t \in I} \|\widetilde{U}^{(r+1)}(t)\|
\leq \sup_{t \in I} \big\|(\widetilde{U}-u)^{(r+1)}(t)\big\|
+ \sup_{t \in I} \|u^{(r+1)}(t)\|
\leq C(F,u) \tau + C(u).
\end{gather*}
Because of $0 \leq k \leq r$ this completes the proof of the first statement.
The second estimate can be proven similarly.
\end{proof}

\begin{corollary}[Superconvergence in time mesh points]
\label{cor:highSuperconv}
Let the error estimates of Proposition~\ref{prop:collErrorTimePoints}
hold for the solution $\widetilde{U}$ of~\eqref{eq:collocation} and the exact
solution $u$ of~\eqref{initValueProb}. Then we have
\begin{gather*}
\max_{1 \leq n \leq N} \big\|(u-U)(t_n^-)\big\|
                      = \max_{1 \leq n \leq N} \big\|(u-\widetilde{U})(t_n^-)\big\| \leq C(F,u) \tau^{2r-k+1}
\end{gather*}
for the solution $U$ of $\Q{r}{k}$-$\vtd{r}{k}$.
\end{corollary}
\begin{proof}
Recall that $U = \I{r}{k} \widetilde{U}$ and that $\I{r}{k}$ especially preserves the
function value in $t_n^-$. Hence, $U(t_n^-) = \I{r}{k} \widetilde{U}(t_n^-) = \widetilde{U}(t_n^-)$
and the estimate follows immediately from Proposition~\ref{prop:collErrorTimePoints}.
\end{proof}

\begin{remark}[Superconvergence in quadrature points]
We obtain under the assumptions of Corollary~\ref{cor:globalErrorEst} also
a (lower order) superconvergence estimate for the solution $U$ of
$\Q{r}{k}$-$\vtd{r}{k}$ in the quadrature points of $\Q{r}{k}$ if $0 \leq k
< r$. In fact, let $t_{n,i}$, $i = 1,\ldots,r-k$, denote the local
quadrature points of $\Q{r}{k}$ in the interior of $I_n$. Then, we have
\begin{gather*}
\big\| (u-U)(t_{n,i}) \big\|
= \big\| (u-\widetilde{U})(t_{n,i})\big\|
\leq C(F,u) \tau^{(r+1)+1}.
\end{gather*}
In addition, we obtain
\begin{gather*}
\big\| (u-U)^{(l)}(t_n^-) \big\|
= \big\| (u-\widetilde{U})^{(l)}(t_n^-)\big\|
\leq C(F,u) \tau^{(r+1)+1-l}, \qquad 0\le l \le  \lfloor
\tfrac{k}{2}\rfloor,
\end{gather*}
and
\begin{gather*}
\big\| (u-U)^{(l)}(t_{n-1}^+) \big\|
= \big\| (u-\widetilde{U})^{(l)}(t_{n-1}^+)\big\|
\leq C(F,u) \tau^{(r+1)+1-l}, \qquad 0\le l \le  \lfloor
\tfrac{k-1}{2}\rfloor,
\end{gather*}
provided $k\ge 1$.

These superconvergence estimates especially imply
\begin{gather*}
\left(\sum_{n=1}^N \Q{r}{k}\Big[ \|u-U\|^2\Big]_{I_n}\right)^{\!\!1/2}
= \left(\sum_{n=1}^N \Q{r}{k}\Big[ \|u-\widetilde{U}\|^2\Big]_{I_n}\right)^{\!\!1/2}
\leq (t_N-t_0)^{1/2} C(F,u) \tau^{(r+1)+1}
\end{gather*}
which compared to
\begin{gather*}
\left(\sum_{n=1}^N \int_{I_n} \big\|(u-U)(x)\big\|^2 \d x\right)^{\!\!1/2}
\leq (t_N-t_0)^{1/2} C(F,u) \tau^{r+1}
\end{gather*}
gives an extra order of convergence.
\end{remark}

\begin{remark}[Superconconvergence of derivative(s) in time mesh points]
From the point conditions~\eqref{eq:locProbE} and the
bound of Corollary~\ref{cor:highSuperconv} we also gain superconvergence
estimates up to the $\left\lfloor \frac{k}{2} \right\rfloor$th
derivative of the solution $U$ of $\Q{r}{k}$-$\vtd{r}{k}$ in $t_n^-$,
provided that $F$ satisfies certain Lipschitz-conditions. Indeed,
supposing that
\begin{gather*}
\left\|\frac{\d^i}{\d t^i} \Big(F\big(t,v(t)\big) - F\big(t,w(t)\big)\Big)\Big|_{t=t_n^-}\right\|
\leq C \sum_{j=0}^i \big\|(v-w)^{(j)}(t_n^-)\big\|,
\end{gather*}
holds for all $0 \leq i \leq \left\lfloor\frac{k}{2}\right\rfloor-1$, 
we obtain
\begin{align*}
\big\|(u-U)^{(i+1)}(t_n^-)\big\|
= \left\|\frac{\d^i}{\d t^i} \Big(F\big(t,u(t)\big) - F\big(t,U(t)\big)\Big)\Big|_{t=t_n^-}\right\|
  & \leq C \sum_{j=0}^i \big\|(u-U)^{(j)}(t_n^-)\big\| \\
  & \leq C(F,u) \tau^{2r-k+1}
\end{align*}
by iteration over $i=0,\ldots,\left\lfloor\frac{k}{2}\right\rfloor-1$.
\end{remark}

\section{Interpolation cascade}
\label{sec:intCascade}

This section is restricted to study affine linear problems of the form

\medskip

Find $u : \overline{I} \to \RR^d$ such that
\begin{gather}
\label{initValueProbLinear}
M u'(t) = f(t)- A u(t), \qquad u(t_0) = u_0 \in \RR^d,
\end{gather}
where $M, A \in \RR^{d \times d}$ are time-independent matrices and $M$ is regular,
i.e., in the general setting we have $F(t,u) = f(t)-Au$.

\subsection{A slight modification of the method}

Let $0\le k\le r$. In order to solve~\eqref{initValueProbLinear} numerically,
we define the $\II$-$\vtd{r}{k}\big(g\big)$ problem by
\medskip

Given $U(t_{n-1}^-)$, find $U\in P_r(I_n,\RR^d)$ such that
\begin{subequations}
\begin{align}
U(t_{n-1}^+)
& = U(t_{n-1}^-),
&& \text{if } k \geq 1, \\
M U^{(i+1)}(t_n^-)
& = g^{(i)}(t_n^-)-AU^{(i)}(t_n^-),
&& \text{if } k \geq 2, \, i = 0, \ldots, \left\lfloor\tfrac{k}{2}\right\rfloor - 1, 
\label{eq:linlocProbE} \\
M U^{(i+1)}(t_{n-1}^+)
& = g^{(i)}(t_{n-1}^+)-AU^{(i)}(t_{n-1}^+),
&& \text{if } k \geq 3, \, i = 0, \ldots, \left\lfloor\tfrac{k-1}{2}\right\rfloor - 1,
\label{eq:linlocProbA}
\end{align}
and
\begin{equation}
\label{eq:linlocProbVar}
\IIn{ \big(M U',\varphi\big) }
+ \delta_{0,k} \big(M \big[U\big]_{n-1},\varphi(t_{n-1}^+) \big)
= \IIn{ \big(g - AU, \varphi\big) }
\qquad\forall \varphi\in P_{r-k}(I_n,\RR^d)
\end{equation}
\end{subequations}
where $U(t_0^-) = u_0$ and $g$ will be chosen later on depending on $f$. As before
$\II$ denotes an integrator, typically the integral over $I_n$ or a
quadrature formula.

Recall that $\Q{r}{k}$ denotes the quadrature rule associated
to $\vtd{r}{k}$ determined by~\eqref{quadrule:Qpp}. This quadrature rule is exact
for polynomials up to degree $2r-k$. Furthermore, $\I{r}{k}$ is the Hermite
interpolation associated to the quadrature rule $\Q{r}{k}$.

In a first step, we will consider $\Q{r}{k}$-$\vtd{r}{k}\big(\I{r+1}{k+2} f\big)$
for $0 \leq k \leq r-1$. Note that the case $r=k$ needs to be excluded
since otherwise $\I{r+1}{k+2} f$ would not be well defined.

In view of the postprocessing of Section~\ref{sec:postprocessing} the modified
method has some interesting properties as we shall show now.

\begin{theorem}
\label{th:modPostproc}
Let $r, k \in \ZZ$, $0 \leq k \leq r-1$. Suppose that $U \in \Y{r}$ solves
$\Q{r}{k}$-$\vtd{r}{k}\big(\I{r+1}{k+2}f\big)$.
Determine $\widetilde{U} \in \Y{r+1}$ by the postprocessing
of Theorem~\ref{th:postproc} or Corollary~\ref{cor:postprocAlt},
respectively. Then $\widetilde{U}$ solves $\Q{r+1}{k+2}$-$\vtd{r+1}{k+2}\big(f\big)$.
\end{theorem}
\begin{proof}
First of all, note that we also could use $\I{r+1}{k+2}f$ instead of $f$ in
the definition of the correction vector $a_n \in \RR^d$ for the
postprocessing, see~\eqref{eq:an}, since $\I{r+1}{k+2}$ preserves all
occurring point values.  Hence, postprocessing $U$ as in
Section~\ref{sec:postprocessing} yields a function $\widetilde{U} \in
\Y{r+1}$ which solves $\Q{r}{k}$-$\vtd{r+1}{k+2}\big(\I{r+1}{k+2}f\big)$,
i.e., $\widetilde{U} \in P_{r+1}(I_n,\RR^d)$ satisfies locally (on $I_n$)
\begin{subequations}
\begin{align}
\widetilde{U}(t_{n-1}^+)
& = \widetilde{U}(t_{n-1}^-), \\
M\widetilde{U}^{(i+1)}(t_n^-)
& = (\I{r+1}{k+2}f)^{(i)}(t_n^-)-A\widetilde{U}^{(i)}(t_n^-),
&& i = 0, \ldots, \left\lfloor\tfrac{k}{2}\right\rfloor\!,
\label{eq:linlocProbIntEPost} \\
M\widetilde{U}^{(i+1)}(t_{n-1}^+)
& = (\I{r+1}{k+2}f)^{(i)}(t_{n-1}^+)-A\widetilde{U}^{(i)}(t_{n-1}^+),
&& \text{if } k \geq 1, \, i = 0, \ldots, \left\lfloor\tfrac{k-1}{2}\right\rfloor\!,
\label{eq:linlocProbIntAPost}
\intertext{and}
\label{eq:linlocProbIntVarPost}
\Q{r}{k} \Big[ \big(M\widetilde{U}',\varphi\big)\Big]
& = \Q{r}{k} \Big[ \big(\I{r+1}{k+2}f - A\widetilde{U}, \varphi\big)\Big]
&& \forall \varphi\in P_{r-k-1}(I_n,\RR^d)
\end{align}
\end{subequations}
where $\widetilde{U}(t_0^-) = u_0$.

Also in~\eqref{eq:linlocProbIntEPost} and~\eqref{eq:linlocProbIntAPost}
the interpolation operator $\I{r+1}{k+2}$
preserves all occurring point values and therefore could be dropped. Moreover, we see
that only polynomials of maximal degree $2r-k$ appear in~\eqref{eq:linlocProbIntVarPost}.
Since both quadrature rules $\Q{r}{k}$ and $\Q{r+1}{k+2}$ are exact in this case
and the interpolation operator $\I{r+1}{k+2}$ uses the quadrature points of $\Q{r+1}{k+2}$, we obtain that
\begin{align*}
& \Q{r}{k}\Big[\big(M\widetilde{U}'+A\widetilde{U}-\I{r+1}{k+2} f,\varphi\big)\Big]
= \int_{I_n} \!\big(M\widetilde{U}'(t)+A\widetilde{U}(t)-\I{r+1}{k+2} f(t),\varphi(t)\big) \, \d t \\
& \quad = \Q{r+1}{k+2}\Big[\big(M\widetilde{U}'+A\widetilde{U}-\I{r+1}{k+2} f,\varphi\big)\Big]
= \Q{r+1}{k+2}\Big[\big(M\widetilde{U}'+A\widetilde{U}-f,\varphi\big)\Big]
\qquad \forall \varphi \in P_{r-k-1}(I_n,\RR^d).
\end{align*}
Summarizing, we have seen that the postprocessed solution $\widetilde{U}$ of
$\Q{r}{k}$-$\vtd{r}{k}\big(\I{r+1}{k+2}f\big)$ solves $\Q{r+1}{k+2}$-$\vtd{r+1}{k+2}\big(f\big)$.
\end{proof}

\begin{remark}
Within the above argument we proved that
the method $\Q{r+1}{k+2}$-$\vtd{r+1}{k+2}\big(f\big)$
and the method $\Q{r}{k}$-$\vtd{r+1}{k+2}\big(\I{r+1}{k+2} f\big)$
are equivalent for $0 \leq k \leq r-1$.

Similarly, one can show that the method $\Q{r}{k}$-$\vtd{r+1}{k+2}\big(f\big)$
and the method $\Q{r+1}{k+2}$-$\vtd{r+1}{k+2}\big(\I{r}{k} f\big)$ are
equivalent for $0 \leq k \leq r-1$.
Note that also $\I{r}{k}$ preserves all derivatives that appear
in the point conditions at both ends of the interval.
\end{remark}

\subsection{Interpolation cascade}

Having a closer look at the result of Theorem~\ref{th:modPostproc}, we see
that the postprocessed solution of the modified discrete problem also solves
a numerically integrated variational time discretization method but with the
``right'' associated quadrature rule. This enables to do one further postprocessing
step.

For $1 \leq \ell \leq r-k$, using an interpolation cascade we even could enable up to $\ell + 1$ additional
postprocessing steps. More concretely we have
\begin{align*}
& \text{$\Q{r}{k}$-$\vtd{r}{k}\big(\I{r+1}{k+2} \circ \I{r+2}{k+4} \circ \ldots \circ \I{r+\ell}{k+2\ell}f\big)$}
\leadsto \text{$\Q{r+1}{k+2}$-$\vtd{r+1}{k+2}\big(\I{r+2}{k+4} \circ \ldots \circ \I{r+\ell}{k+2\ell}f\big)$} \\
  \leadsto \quad & \ldots \\
  \leadsto \quad
& \text{$\Q{r+\ell-1}{k+2(\ell-1)}$-$\vtd{r+\ell-1}{k+2(\ell-1)}\big(\I{r+\ell}{k+2\ell}f\big)$}
\leadsto \text{$\Q{r+\ell}{k+2\ell}$-$\vtd{r+\ell}{k+2\ell}\big(f\big)$} \\
  \leadsto \quad
& \text{$\Q{r+\ell}{k+2\ell}$-$\vtd{r+\ell+1}{k+2(\ell+1)}\big(f\big)$}
\end{align*}
where $\leadsto$ denotes the postprocessing steps.
Note that $f$ itself can be used in each postprocessing step to calculate
the correction vector $a_n \in \RR^d$ (cf.~Theorem~\ref{th:postproc}) since
in each step the occurring derivative of $f$ at $t_n^-$ is preserved by the
respective interpolation cascade.  

\begin{remark}
For Dahlquist's stability equation, i.e., $d=1$, $M = 1$, $A = -\lambda$,
and $f = 0$ in~\eqref{initValueProbLinear}, we easily see that
\begin{gather*}
\vtd{r-\ell}{k-2\ell}\big(f\big) =
\Q{r-\ell}{k-2\ell}\text{-}\vtd{r-\ell}{k-2\ell}\big(f\big) = 
\Q{r-\ell}{k-2\ell}\text{-}\vtd{r-\ell}{k-2\ell}\big(\I{r-\ell+1}{k-2\ell+2} \circ \I{r-\ell+2}{k-2\ell+4} \circ \ldots \circ \I{r}{k}f\big)
\end{gather*}
for all $\ell = 0,\ldots,\left\lfloor \frac{k}{2} \right\rfloor$.
Thus, $\ell$ postprocessing steps can be applied for this equation. Since
the postprocessing does not change the function value in the end points of
the intervals, the stability function does not change either.
Therefore, $\vtd{r}{k}$ as well as $\Q{r}{k}\text{-}\vtd{r}{k}$ provide
the same stability function as $\vtd{r-\ell}{k-2\ell}$. With the special
choice $\ell = \left\lfloor \frac{k}{2} \right\rfloor$, we immediately
find that $\vtd{r}{k}$ shares its stability properties with
\begin{gather*}
\vtd{r-\left\lfloor \frac{k}{2} \right\rfloor}{k-2\left\lfloor \frac{k}{2}
\right\rfloor} = 
\begin{cases}
\vtd{r-\left\lfloor \frac{k}{2} \right\rfloor}{0}
= \mathrm{dG}\!\left(r-\left\lfloor \frac{k}{2} \right\rfloor\right),		& \text{if $k$ is even}, \\
\vtd{r-\left\lfloor \frac{k}{2} \right\rfloor}{1}
= \mathrm{cGP}\!\left(r-\left\lfloor \frac{k}{2} \right\rfloor\right),
& \text{if $k$ is odd}.
\end{cases}
\end{gather*}
Hence, all methods with even $k$ share their strong A-stability with the dG method while methods
with odd $k$ are A-stable as the cGP method, cf.~Remark~\ref{rem:dGcGP} and~\cite{ABBM19}.
\end{remark}

\begin{remark}
Analogously to Theorem~\ref{th:postproc} also a postprocessing
from $\Q{r}{k}$-$\vtd{r}{k}(g)$ to $\Q{r}{k}$-$\vtd{r+1}{k+2}(g)$
can be proven when instead of~\eqref{eq:an} the correction vector
is determined by
\begin{equation*}
\check{a}_n = M^{-1} \left(
g^{\left(\left\lfloor \frac{k}{2} \right\rfloor\right)}(t_n^-)
- A U^{\left(\left\lfloor \frac{k}{2} \right\rfloor\right)}(t_n^-)
- M U^{\left(\left\lfloor \frac{k}{2} \right\rfloor+1\right)}(t_n^-)\right)\!.
\end{equation*}
However, when $g$ and its derivatives are not globally continuous up to a
sufficiently high order, in general the discrete solution of $\Q{r}{k}$-$\vtd{r}{k}(g)$
is not $\lfloor \frac{k-1}{2} \rfloor$-times continuously differentiable.
Therefore, the postprocessing by jumps and the postprocessing by
(modified) residuals will not provide the same correction anymore.

A more detailed analysis shows that applying two postprocessing steps
based on residuals on the solution of $\Q{r}{k}$-$\vtd{r}{k}(f)$ yields
the solution of $\Q{r+2}{k+4}$-$\vtd{r+2}{k+4}(\If^{r+1,*} f)$ where
$\If^{r+1,*} f$ interpolates $f$ in the quadrature points of $\Q{r}{k}$
and additionally preserves its ($\lfloor \frac{k}{2} \rfloor +1$)th
derivative in $t_n^-$. Similarly for dG-like methods (characterized by
even $k$) we find that applying two postprocessing steps based on jumps
on the solution of $\Q{r}{k}$-$\vtd{r}{k}(f)$ gives the solution of
$\Q{r+2}{k+4}$-$\vtd{r+2}{k+4}(\If^{r+1}_* f)$ where $\If^{r+1}_* f$
interpolates $f$ in the quadrature points of $\Q{r}{k}$ and additionally
preserves its ($\lfloor \frac{k-1}{2} \rfloor +1$)th derivative
in $t_{n-1}^+$.
\end{remark}

\section{Derivatives of solutions}
\label{sec:nestedSol_Der}

As in Section~\ref{sec:intCascade} we consider affine linear problems
of the form~\eqref{initValueProbLinear} with time-independent coefficients.
Since the quadrature formula $\Q{r}{k}$ is exact for polynomials up to degree $2r-k$
and the associated interpolation operator $\I{r}{k}$ yields polynomials of degree $r$, we can
write both $\vtd{r}{k}\big(f\big)$ and $\Q{r}{k}$-$\vtd{r}{k}\big(f\big)$ for $0 \leq k \leq r$ in
the form
\medskip

Find $U \in P_{r}(I_n,\RR^d)$ for all $n=1,\ldots,N$ such that
\begin{subequations}
\label{eq:nested}
\begin{align}
U(t_{n-1}^+)
& = U(t_{n-1}^-),
&& \text{if } k \geq 1, \label{eq:nestedCont} \\
M U^{(i+1)}(t_{n}^-)
& = (\If f)^{(i)}(t_{n}^-) - A U^{(i)}(t_{n}^-),
&& \text{if } k \geq 2, \, i = 0,\ldots,\left\lfloor \tfrac{k}{2} \right\rfloor -1, \label{eq:nestedE} \\
M U^{(i+1)}(t_{n-1}^+)
& = (\If f)^{(i)}(t_{n-1}^+) - A U^{(i)}(t_{n-1}^+),
&& \text{if } k \geq 3, \, i = 0,\ldots,\left\lfloor \tfrac{k-1}{2} \right\rfloor -1, \label{eq:nestedA}
\end{align}
as well as
\begin{gather}
\int_{I_n} \big(M U'+AU,\varphi\big) \,\d t + \delta_{0,k} \big(M \big[U\big]_{n-1}, \varphi(t_{n-1}^+)\big)
= \int_{I_n} \big(\If f,\varphi\big) \,\d t
\quad \forall \varphi \in P_{r-k}(I_n,\RR^d) \label{eq:nestedVar}
\end{gather}
\end{subequations}
with $U(t_0^-) = u_0$
where $\If = \mathrm{Id}$ for $\vtd{r}{k}\big(f\big)$ and $\If = \I{r}{k}$ for $\Q{r}{k}$-$\vtd{r}{k}\big(f\big)$,
respectively. Note that it is still ensured that $U$ is
$\left\lfloor \frac{k-1}{2} \right\rfloor$-times continuously differentiable.
Moreover, since the operator $\If \in \{\Id,\I{r}{k}\}$ keeps at $t_n^-$ and
$t_{n-1}^+$ derivatives up to order $\left\lfloor \frac{k}{2} \right\rfloor$ and
$\left\lfloor \frac{k-1}{2} \right\rfloor$, respectively, we could drop $\If$
in~\eqref{eq:nestedE} and~\eqref{eq:nestedA}.

\begin{theorem}
\label{th:varForJthDer}
Let $r, k \in \ZZ$, $0 \leq k \leq r$, and suppose that $U \in \Y{r}$
solves $\vtd{k}{r}\big(\If f\big)$ with $\If \in \{\Id,\I{r}{k}\}$.
Then it holds
\begin{gather*}
\int_{I_n} \big(M(U^{(j)})'+AU^{(j)},\varphi\big) \,\d t
+ \delta_{0,k-2j} \big(M \big[U^{(j)}\big]_{n-1}, \varphi(t_{n-1}^+)\big)
= \int_{I_n} \big((\If f)^{(j)},\varphi\big)\, \d t
\end{gather*}
for all $0 \leq j \leq \left\lfloor \frac{k}{2} \right\rfloor$ and all
$\varphi \in P_{r-k+j}(I_n,\RR^d)$. Note that the integrals can be replaced
for $\If = \I{r}{k}$ by any quadrature rule which is exact for polynomials of degree less
than or equal to $2r-k$, for example by $\Q{r-j}{k-2j}$.
\end{theorem}
\begin{proof}
First of all, we consider the case $0 \leq j \leq \left\lfloor \frac{k-1}{2} \right\rfloor$.
Integrating by parts several times and using~\eqref{eq:nested}, we gain for
any $\varphi \in P_{r-k+j}(I_n,\RR^d)$
\begin{align*}
& \int_{I_n} \big(M(U^{(j)})'+AU^{(j)},\varphi\big) \,\d t
= \int_{I_n} \big(\partial_t^j(MU'+AU),\varphi\big) \,\d t \\
& \quad = - \int_{I_n} \big(\partial_t^{j-1}(MU'+AU),\varphi'\big) \,\d t 
+ \big[\big(\partial_t^{j-1}(MU'+AU),\varphi\big)\big]_{t_{n-1}^+}^{t_n^-} \\
& \quad = \int_{I_n} \big(\partial_t^{j-2}(MU'+AU),\varphi''\big) \,\d t
- \big[\big(\partial_t^{j-2}(MU'+AU),\varphi'\big)\big]_{t_{n-1}^+}^{t_n^-}
+ \big[\big(\partial_t^{j-1}(MU'+AU),\varphi\big)\big]_{t_{n-1}^+}^{t_n^-} \\
& \quad = \ldots \\
& \quad = (-1)^j \int_{I_n} \big(MU'+AU,\underbrace{\varphi^{(j)}}_{\in P_{r-k}(I_n,\RR^d)}\big) \,\d t
+ \sum_{l=0}^{j-1} (-1)^l \big[\big(\partial_t^{j-1-l}(MU'+AU),
\varphi^{(l)}\big)\big]_{t_{n-1}^+}^{t_n^-} \\
& \quad = (-1)^j \int_{I_n} \big(\If f,\varphi^{(j)}\big) \,\d t
+ \sum_{l=0}^{j-1} (-1)^l \big[\big((\If f)^{(j-1-l)},\varphi^{(l)}\big)\big]_{t_{n-1}^+}^{t_n^-} \\
& \quad = \ldots = \int_{I_n} \big((\If f)^{(j)},\varphi\big)\, \d t
\end{align*}
which is the desired statement. So, for odd $k$ we are done due to
$\left\lfloor\frac{k-1}{2}\right\rfloor = \left\lfloor\frac{k}{2}\right\rfloor$
in this case.

Hence, it only remains to study the case $j=\left\lfloor \frac{k}{2} \right\rfloor$ for even $k\ge 2$.
Similar as above
we conclude from~\eqref{eq:nested} for any $\varphi \in P_{r-k+j}(I_n,\RR^d)$ that
\begin{align*}
& \int_{I_n} \big(M(U^{(j)})'+AU^{(j)},\varphi\big) \,\d t \\
& \quad = (-1)^j \int_{I_n} \big(MU'+AU,\varphi^{(j)}\big) \,\d t
+ \sum_{l=0}^{j-1} (-1)^l \big[\big(\partial_t^{j-1-l}(MU'+AU),
\varphi^{(l)}\big)\big]_{t_{n-1}^+}^{t_n^-} \\
& \quad = (-1)^j \int_{I_n} \big(\If f,\varphi^{(j)}\big) \,\d t
+ \sum_{l=0}^{j-1} (-1)^l \big[\big((\If f)^{(j-1-l)},\varphi^{(l)}\big)\big]_{t_{n-1}^+}^{t_n^-} \\
& \quad \qquad - \big(MU^{(j)}(t_{n-1}^+)+AU^{(j-1)}(t_{n-1}^+),\varphi(t_{n-1}^+)\big)
+ \big((\If f)^{(j-1)}(t_{n-1}^+),\varphi(t_{n-1}^+)\big) \\
& \quad = \int_{I_n} \big((\If f)^{(j)},\varphi\big)\, \d t
- \big(MU^{(j)}(t_{n-1}^+)+AU^{(j-1)}(t_{n-1}^+),\varphi(t_{n-1}^+)\big)
+ \big((\If f)^{(j-1)}(t_{n-1}^+),\varphi(t_{n-1}^+)\big).
\end{align*}
Since point values of $\If f$ and $f$ only appear at $t_n^-$ and $t_{n-1}^+$ up
to the $(\left\lfloor \frac{k}{2} \right\rfloor -1)$th derivative which the operator
$\If \in \{\Id,\I{r}{k}\}$ preserves, we obtain,
using the continuity of $f^{(j-1)}$ and~\eqref{eq:nestedE}, for $n \geq 2$
\begin{multline*}
\big((\If f)^{(j-1)}(t_{n-1}^+),\varphi(t_{n-1}^+)\big)
= \big(f^{(j-1)}(t_{n-1}^+),\varphi(t_{n-1}^+)\big)
= \big(f^{(j-1)}(t_{n-1}^-),\varphi(t_{n-1}^+)\big) \\
= \big((\If f)^{(j-1)}(t_{n-1}^-),\varphi(t_{n-1}^+)\big)
= \big(MU^{(j)}(t_{n-1}^-)+AU^{(j-1)}(t_{n-1}^-),\varphi(t_{n-1}^+)\big).
\end{multline*}
Since also $U^{(j-1)}$ is continuous, we get
\begin{gather*}
\big(MU^{(j)}(t_{n-1}^+)+AU^{(j-1)}(t_{n-1}^+),\varphi(t_{n-1}^+)\big)
- \big((\If f)^{(j-1)}(t_{n-1}^+),\varphi(t_{n-1}^+)\big)
= \big( M\big[U^{(j)}\big]_{n-1},\varphi(t_{n-1}^+)\big)
\end{gather*}
with $\big[U^{(j)}\big]_{n-1} = U^{(j)}(t_{n-1}^+)-U^{(j)}(t_{n-1}^-)$.
Thus, we gain
\begin{align*}
& \int_{I_n} \big(M(U^{(j)})'+AU^{(j)},\varphi\big) \,\d t \\
& \quad = \int_{I_n} \big((\If f)^{(j)},\varphi\big)\, \d t
  - \begin{cases}
\big( M \big[U^{(j)}\big]_{n-1},\varphi(t_{n-1}^+)\big), & n \geq 2, \\
\big( M U^{(j)}(t_0^+)-((\If f)^{(j-1)}(t_0^+)-AU^{(j-1)}(t_0^+)),\varphi(t_0^+)\big), & n=1,
\end{cases}
\end{align*}
for all $\varphi \in P_{r-k+j}(I_n,\RR^d)$. Recalling the definitions
of $U^{(j)}(t_0^-)$ and $u^{(j)}(t_0)$ we find
\begin{align*}
(\If f)^{(j-1)}(t_0^+)-AU^{(j-1)}(t_0^+)
  & = f^{(j-1)}(t_0^+)-A U^{(j-1)}(t_0^-) \\
  & = f^{(j-1)}(t_0^+)-A u^{(j-1)}(t_0)
= M u^{(j)}(t_0) = M U^{(j)}(t_0^-)
\end{align*}
which completes the proof.
\end{proof}

Using the appropriate initial condition, derivatives of $\mathbf{VTD}$
solutions are themselves solutions of $\mathbf{VTD}$ methods.
\begin{corollary}
Let $r, k \in \ZZ$, $0 \leq k \leq r$, and suppose that $U \in \Y{r}$
solves $\vtd{r}{k}\big(\If f\big)$ where $\If \in \{\mathrm{Id}, \I{r}{k}\}$.
Then $U^{(j)} \in \Y{r-j}$, $0 \leq j \leq \left\lfloor \frac{k}{2} \right\rfloor$, solves
$\vtd{r-j}{k-2j}\big((\If f)^{(j)}\big)$ if $u^{(j)}(t_0)$ is used as initial condition.
\end{corollary}
\begin{proof}
Because of Theorem~\ref{th:varForJthDer} it only remains to prove the needed
conditions at $t_{n-1}^+$ and $t_n^-$. Since we have by construction that $U$
is $\left\lfloor \frac{k-1}{2} \right\rfloor$-times continuously differentiable, the
desired identities follow from the fact that $U^{(j)}$ is continuous together
with~\eqref{eq:nestedCont}, \eqref{eq:nestedE}, and~\eqref{eq:nestedA}.
\end{proof}

\section{Numerical experiments}
\label{sec:NumExp}
We will present in this section some numerical tests supporting the
theoretical results. All calculations were carried out using the software
Julia~\cite{julia} using the floating point data type \texttt{BigFloat}
with 512 bits.

\begin{example}
\label{Ex1}
We consider the initial value problem
\begin{equation*}
	\begin{pmatrix}u_1'(t)\\u_2'(t)\end{pmatrix}
	= \begin{pmatrix} -u_1^2(t)-u_2(t)\\ u_1(t)-u_1(t)u_2(t)\end{pmatrix}\!,
	\quad t\in(0,32),\qquad 
	u(0) = \begin{pmatrix} 1/2 \\ 0\end{pmatrix}\!,
\end{equation*}
of a system of nonlinear ordinary differential equations which has
\[
u_1(t) = \frac{\cos t}{2+\sin t},\qquad u_2(t) = \frac{\sin t}{2+\sin t}
\]
as solution.
\end{example}

The appearing nonlinear systems within each time step were solved by
Newton's method where we applied a Taylor expansion of the inherited data
from the previous time interval to calculate an initial guess for all
unknowns on the current interval. If higher order derivatives were needed
at initial time $t=0$, the ode system and its temporal derivatives were
used, see~\eqref{eq:IC}. The postprocessing used the jumps of the
derivatives, as given in Corollary~\ref{cor:postprocAlt}.

We denote by
\[
e := u-U, \qquad \tilde{e} := u-\widetilde{U}
\]
the error of the solution $U$ and the error of the postprocessed solution
$\widetilde{U}$, respectively. Errors were measured in the norms
\[
\|\varphi\|_{L^2} := \left(\int_{t_0}^{t_N} \|\varphi(t)\|^2\,\d t\right)^{\!1/2},\qquad
\|\varphi\|_{\ell^\infty} := \max_{1\le n\le N} \|\varphi(t_n^-)\|
\]
where $\|\cdot\|$ denotes the Euclidean norm in $\RR^d$.

\begin{table}[htb!]
\caption{Example~\ref{Ex1}: Results for $\Q{6}{0}$-$\vtd{6}{0}=\mathrm{dG}(6)$.
\label{Ex1:vtd60}}
\centerline{%
\begin{tabular}{rccccccc}
\toprule
N
& $\|e\|_{L^2}$ & $\|e\|_{\ell^\infty}$ & $\|\tilde{e}\|_{L^2}$
& $\|e'\|_{L^2}$ & $\|e'\|_{\ell^\infty}$ & $\|\tilde{e}'\|_{L^2}$ & $\|\tilde{e}'\|_{\ell^\infty}$ \\
\midrule
 128 & 3.3024-09 & 1.0930-17 & 2.4964-10 & 4.8620-07 & 2.2496-07 & 1.9306-08 & 1.2577-17\\
 256 & 2.6073-11 & 1.3846-21 & 9.8983-13 & 7.6991-09 & 3.5726-09 & 1.5313-10 & 1.5217-21\\
 512 & 2.0424-13 & 1.6851-25 & 3.8808-15 & 1.2070-10 & 5.6046-11 & 1.2008-12 & 1.8512-25\\
1024 & 1.5967-15 & 2.0544-29 & 1.5174-17 & 1.8876-12 & 8.7659-13 & 9.3902-15 & 2.2580-29\\
2048 & 1.2476-17 & 2.5064-33 & 5.9286-20 & 2.9500-14 & 1.3700-14 & 7.3378-17 & 2.7557-33\\
4096 & 9.7473-20 & 3.0587-37 & 2.3160-22 & 4.6096-16 & 2.1408-16 & 5.7330-19 & 3.3631-37\\
8192 & 7.6151-22 & 3.7333-41 & 9.0469-25 & 7.2025-18 & 3.3450-18 & 4.4790-21 & 4.1049-41\\
\midrule
eoc  & 7.00      & 13.00     & 8.00      & 6.00      & 6.00      & 7.00      &13.00\\
theo & 7         & 13        & 8         & 6         & 6         & 7         &13   \\
\bottomrule
\end{tabular}
}
\end{table}
Table~\ref{Ex1:vtd60} presents the results for $\Q{6}{0}$-$\vtd{6}{0}$ which
is just $\mathrm{dG}(6)$ with numerical quadrature by the right-sided
Gauss--Radau formula with $7$ points. We show norms of the error between
the solution $u$ and the discrete solution $U$ as well as the error between
the solution $u$ and the postprocessed discrete solution $\widetilde{U}$ in
different norms. Using the results for $N=4096$ and $N=8192$, the
experimental order of convergence (eoc) is calculated. In addition, the
theoretically predicted convergence orders (theo) are given. We see clearly
from Table~\ref{Ex1:vtd60} that the experimental orders of convergence
coincide with the theoretical predictions. This holds for the function
itself and its time derivative. Moreover, the order of convergence
increases by $1$ if one postprocessing step is applied. It is noteworthy
that the error norm $\|\tilde{e}'\|_{\ell^{\infty}}$ shows the same high
order superconvergence order as $\|e\|_{\ell^{\infty}}$. This behavior is
due to the collocation conditions satisfied by the postprocessed solution
$\widetilde{U}$.

\begin{table}[htb!]
\caption{Example~\ref{Ex1}: Results for $\Q{6}{5}$-$\vtd{6}{5}$.
\label{Ex1:vtd65}}
\centerline{%
\begin{tabular}{rccccccc}
\toprule
N & $\|e\|_{L^2}$ & $\|e\|_{\ell^{\infty}}$ & $\|\tilde{e}\|_{L^2}$
& $\|e'\|_{L^2}$ & $\|e'\|_{\ell^{\infty}}$ & $\|\tilde{e}'\|_{L^2}$ & $\|\tilde{e}'\|_{\ell^{\infty}}$ \\
\midrule
 128 & 3.7426-08 & 1.1561-09 & 1.2404-08 & 1.0494-06 & 1.6575-09 & 2.0501-07 & 1.6576-09\\
 256 & 2.8282-10 & 4.5523-12 & 5.0078-11 & 1.6409-08 & 6.3612-12 & 1.6318-09 & 6.3612-12\\
 512 & 2.1881-12 & 1.7984-14 & 1.9722-13 & 2.5641-10 & 2.5044-14 & 1.2807-11 & 2.5044-14\\
1024 & 1.7052-14 & 7.0168-17 & 7.7197-16 & 4.0064-12 & 9.7667-17 & 1.0017-13 & 9.7667-17\\
2048 & 1.3314-16 & 2.7452-19 & 3.0170-18 & 6.2601-14 & 3.8157-19 & 7.8282-16 & 3.8157-19\\
4096 & 1.0400-18 & 1.0722-21 & 1.1787-20 & 9.7814-16 & 1.4907-21 & 6.1162-18 & 1.4907-21\\
8192 & 8.1243-21 & 4.1884-24 & 4.6044-23 & 1.5284-17 & 5.8231-24 & 4.7784-20 & 5.8231-24\\
\midrule
eoc  & 7.00      & 8.00     & 8.00      & 6.00      & 8.00      & 7.00      &8.00\\
theo & 7         & 8        & 8         & 6         & 8         & 7         &8\\
\bottomrule
\end{tabular}
}
\end{table}
The results of our calculations using the variational time discretization
$\Q{6}{5}$-$\vtd{6}{5}$ are collected in Table~\ref{Ex1:vtd65}. Again
we present the
results in different norms for both the error itself and the error obtained
after postprocessing the discrete solution. Also for this temporal
discretization, all theoretically predicted orders of convergence are met by
our numerical experiments. Compared to the results of
$\Q{6}{0}$-$\vtd{6}{0}$ the
superconvergence order measured in $\|\cdot\|_{\ell^{\infty}}$ is much
smaller which is in agreement with our theory. In addition, the order of
convergence of $\|\tilde{e}'\|_{\ell^{\infty}}$ is the same as the order of
convergence of $\|e'\|_{\ell^{\infty}}$ since collocation conditions are
fulfilled already by the discrete solution $U$. Hence, an improvement of
this quantity by applying the postprocessing is not possible.

\begin{table}[htb!]
\caption{Example~\ref{Ex1}: Results for $\Q{6}{6}$-$\vtd{6}{6}$.
\label{Ex1:vtd66}}
\centerline{%
\begin{tabular}{rccccccc}
\toprule
N & $\|e\|_{L^2}$ & $\|e\|_{\ell^{\infty}}$ & $\|\tilde{e}\|_{L^2}$
& $\|e'\|_{L^2}$ & $\|e'\|_{\ell^{\infty}}$ & $\|\tilde{e}'\|_{L^2}$ & $\|\tilde{e}'\|_{\ell^{\infty}}$ \\
\midrule
 128 & 2.5613-07 & 9.1516-08 & 1.4889-07 & 2.6080-06 & 1.1641-07 & 9.5210-07 & 1.1641-07\\
 256 & 2.0921-09 & 7.5844-10 & 1.1839-09 & 3.8709-08 & 8.7360-10 & 7.7532-09 & 8.7350-10\\
 512 & 1.6529-11 & 5.8911-12 & 9.2953-12 & 5.9543-10 & 7.0119-12 & 6.1201-11 & 7.0119-12\\
1024 & 1.2949-13 & 5.5929-14 & 7.2702-14 & 9.2654-12 & 5.4570-14 & 4.7937-13 & 5.4570-14\\
2048 & 1.0123-15 & 3.5852-16 & 5.6810-16 & 1.4462-13 & 4.2568-16 & 3.7475-15 & 4.2568-16\\
4096 & 7.9102-18 & 2.8001-18 & 4.4384-18 & 2.2591-15 & 3.3259-18 & 2.9282-17 & 3.3259-18\\
8192 & 6.1800-20 & 2.1873-20 & 3.4674-20 & 3.5296-17 & 2.5977-20 & 2.2878-19 & 2.5977-20\\
\midrule
eoc  & 7.00      & 7.00     & 7.00      & 6.00      & 7.00      & 7.00      &7.00\\
theo & 7         & 7        & 7         & 6         & 7         & 7         &7\\
\bottomrule
\end{tabular}
}
\end{table}
Table~\ref{Ex1:vtd66} shows the results for calculations using
$\Q{6}{6}$-$\vtd{6}{6}$
as discretization in time. The presented error norms indicate that the
experimental order of convergence are in agreement with our theory. Please
note that the postprocessing does not lead to an improvement of the error
itself. However, there is an improvement if we look at the $L^2$-norm of
the time derivative. We clearly see that the order of convergence is
increased from 6 to 7 which is in agreement with
Proposition~\ref{prop:collErrorGlobal}. Moreover, there is no
superconvergence at the discrete time points, as predicted by our theory.

\begin{example}
\label{Ex2}
We consider the affine linear initial value problem
\[
\begin{pmatrix}
1 & 2\\ -1 & 3
\end{pmatrix}
\begin{pmatrix}
u_1'(t)\\u_2'(t)
\end{pmatrix}
= \begin{pmatrix}
f_1(t) \\ f_2(t)
\end{pmatrix}
- \begin{pmatrix}
1 & 2\\ 3 & 4
\end{pmatrix}
\begin{pmatrix}
u_1(t)\\ u_2(t)
\end{pmatrix}\!,
\quad t\in(0,1),
\qquad
\begin{pmatrix}
u_1(0)\\ u_2(0)
\end{pmatrix}
= \begin{pmatrix}
0\\0
\end{pmatrix}\!,
\]
where $f_1$ and $f_2$ are chosen such that
\[
u_1(t) = (t+t^2) e^t, \qquad u_2(t) = -t e^t
\]
are the solution components.
\end{example}

\begin{table}[htb!]
\caption{Example~\ref{Ex2}: Results for $\Q{7}{0}$-$\vtd{7}{0}=\mathrm{dG}(7)$ 
with cascadic interpolation of $f$
and $s$ postprocessing steps.
\label{Ex2:PP}}
\centerline{%
\begin{tabular}{ccccccccc}
\toprule
& \multicolumn{2}{c}{$\|PP_s e\|_{L^2}$} &
  \multicolumn{2}{c}{$\|(PP_s e)'\|_{L^2}$} &
  \multicolumn{2}{c}{$\|PP_s e\|_{\ell^{\infty}}$} &
  \multicolumn{2}{c}{$\|(PP_s e)'\|_{\ell^{\infty}}$}\\
\cmidrule(lr){2-3}
\cmidrule(lr){4-5}
\cmidrule(lr){6-7}
\cmidrule(lr){8-9}
$s$ & error & order & error & order & error & order & error & order\\
\cmidrule(lr){1-1}
\cmidrule(lr){2-3}
\cmidrule(lr){4-5}
\cmidrule(lr){6-7}
\cmidrule(lr){8-9}
0 & 2.3819-21 & ~8.001 & 5.3781-18 & ~7.001 & 1.4853-40 & 15.002 & 8.738-18 & ~6.983\\
1 & 2.6587-24 & ~9.001 & 2.9547-21 & ~8.001 & 1.4853-40 & 15.002 & 1.010-40 & 15.001\\
2 & 3.6813-27 & 10.001 & 3.2303-24 & ~9.001 & 1.4853-40 & 15.002 & 1.010-40 & 15.001\\
3 & 5.5669-30 & 11.001 & 4.3967-27 & 10.001 & 1.4853-40 & 15.002 & 1.010-40 & 15.001\\
4 & 9.3065-33 & 12.001 & 6.5539-30 & 11.001 & 1.4853-40 & 15.002 & 1.010-40 & 15.001\\
5 & 1.7627-35 & 13.001 & 1.0823-32 & 12.001 & 1.4853-40 & 15.002 & 1.010-40 & 15.001\\
6 & 4.1684-38 & 14.001 & 2.0285-35 & 13.001 & 1.4853-40 & 15.002 & 1.010-40 & 15.001\\
7 & 2.1520-40 & 15.001 & 4.7532-38 & 14.001 & 1.4853-40 & 15.002 & 1.010-40 & 15.001\\
8 & 7.1008-41 & 15.003 & 2.1388-40 & 15.001 & 1.4853-40 & 15.002 & 1.010-40 & 15.001\\
\bottomrule
\end{tabular}
}
\end{table}
Table~\ref{Ex2:PP} presents the results for $\Q{7}{0}$-$\vtd{7}{0}$ where
the cascadic interpolation has been applied to the function $f=(f_1,f_2)$
on the right-hand side, see Section~\ref{sec:intCascade}. We show norms of
the error
$PP_s e$ after $s$ postprocessing steps using $50$ time steps. The given
experimental orders of convergence were calculated from the results with
$25$ and $50$ time steps. Looking at the convergence orders in the
$L^2$-like norms, we clearly see that each postprocessing step increased the
experimental order of convergence by $1$ if at most $7$ postprocessing
steps are applied. The postprocessing step $8$ leads to an improvement of
the convergence order only for the temporal derivative since the function
itself already converges with the optimal order $15$. The postprocessing
has no influence to the $\ell_{\infty}$ norm of the error itself while
the very first postprocessing step improves the results for the derivative
of the error in the $\ell_{\infty}$ norm. This is caused by the fact that
the postprocessed solution fulfills a collocation condition at the discrete
time points.

\begin{table}[htb!]
\caption{Example~\ref{Ex2}: Experimental orders of convergence for $\|(PP_s e)'\|_{L^2}$
using $\Q{7}{k}$-$\vtd{7}{k}$, $k=0,\dots,7$, with cascadic interpolation of $f$,
after $s$ postprocessing steps.}
\label{Ex2:PPall}
\centerline{%
\begin{tabular}{c@{\hspace*{0.9em}}c@{\hspace*{0.9em}}c@{\hspace*{0.9em}}c@{\hspace*{0.9em}}
c@{\hspace*{0.9em}}c@{\hspace*{0.9em}}c@{\hspace*{0.9em}}c@{\hspace*{0.9em}}c@{\hspace*{0.9em}}c}
\toprule
$k$ & $s=0$  & $s=1$  & $s=2$  & $s=3$  & $s=4$  & $s=5$  & $s=6$  & $s=7$  & $s=8$\\
\midrule
0   & 7.001 & 8.001 & 9.001 & 10.001 & 11.001 & 12.001 & 13.001 & 14.001 & 15.001\\
1   & 7.000 & 8.000 & 9.000 & 10.000 & 11.000 & 12.000 & 13.000 & 14.000 &   --- \\
2   & 7.001 & 8.001 & 9.001 & 10.001 & 11.001 & 12.001 & 13.001 &   ---  &   --- \\
3   & 7.000 & 8.000 & 9.000 & 10.000 & 11.000 & 12.000 &   ---  &   ---  &   --- \\
4   & 7.001 & 8.001 & 9.001 & 10.001 & 11.001 &   ---  &   ---  &   ---  &   --- \\
5   & 7.000 & 8.000 & 9.000 & 10.000 &   ---  &   ---  &   ---  &   ---  &   --- \\
6   & 7.002 & 8.002 & 9.002 &   ---  &   ---  &   ---  &   ---  &   ---  &   --- \\
7   & 7.000 & 8.000 &  ---  &   ---  &   ---  &   ---  &   ---  &   ---  &   --- \\
\bottomrule
\end{tabular}
}
\end{table}
Table~\ref{Ex2:PPall} presents the experimental orders of convergence of
$\|(PP_s e)'\|_{L^2}$ for $\Q{7}{k}$-$\vtd{7}{k}$, $k=0,\dots,7$, after $s$
postprocessing steps where at most $r+1-k = 8-k$ steps have been applied. The
cascadic interpolation of the right-hand function $f$ is used for all
considered methods. It can be clearly seen that each additional
postprocessing step increases the convergence by one order. Using the same
number of postprocessing steps, the obtained convergence orders do not
depend on the particular methods. Since each postprocessing step is
covered by our theory and postprocessing by jumps and postprocessing by
residual are equivalent for a single step, both types of postprocessing
lead to identical results if the cascadic interpolation of the right-hand
side function $f$ is used.

\begin{table}[htb!]
\caption{Example~\ref{Ex2}: Experimental orders of convergence for $\|(PP_s e)'\|_{L^2}$
using $\Q{9}{k}$-$\vtd{9}{k}$, $k=0,\dots,9$, and $s$ postprocessing steps based
on jumps, cf.~Corollary~\ref{cor:postprocAlt}.}
\label{Ex2:Jumps}
\centerline{%
\begin{tabular}{c@{\hspace*{0.9em}}c@{\hspace*{0.9em}}c@{\hspace*{0.9em}}c@{\hspace*{0.9em}}
c@{\hspace*{0.9em}}c@{\hspace*{0.9em}}c@{\hspace*{0.9em}}c@{\hspace*{0.9em}}c@{\hspace*{0.9em}}
c@{\hspace*{0.9em}}c@{\hspace*{0.9em}}c}
\toprule
$k$ & $s=0$  & $s=1$  & $s=2$  & $s=3$  & $s=4$  & $s=5$  & $s=6$  & $s=7$  & $s=8$  & $s=9$  & $s=10$\\
\midrule
0   & 9.001 & 10.000 & 11.000 & 10.982 & 10.959 & 10.956 & 10.946 & 10.933 & 10.919 & 10.907 & 10.895\\
1   & 9.000 & 10.000 & 11.000 & ~9.996 & ~9.000 & ~7.998 & ~6.996 & ~5.993 & ~4.990 & ~3.985 &   --- \\
2   & 9.001 & 10.001 & 11.000 & 10.981 & 10.967 & 10.973 & 11.018 & 10.913 & 10.941 &   ---  &   --- \\
3   & 9.000 & 10.000 & 10.002 & ~8.998 & ~7.998 & ~6.997 & ~5.995 & ~4.992 &   ---  &   ---  &   --- \\
4   & 9.001 & 10.001 & 11.000 & 10.983 & 10.960 & 10.955 & 10.944 &   ---  &   ---  &   ---  &   --- \\
5   & 9.000 & 10.000 & 11.000 & ~9.996 & ~9.000 & ~7.998 &   ---  &   ---  &   ---  &   ---  &   --- \\
6   & 9.001 & 10.001 & 11.000 & 10.981 & 10.966 &   ---  &   ---  &   ---  &   ---  &   ---  &   --- \\
7   & 9.000 & 10.000 & 10.007 & ~8.998 &   ---  &   ---  &   ---  &   ---  &   ---  &   ---  &   --- \\
8   & 9.001 & 10.001 & 11.001 &   ---  &   ---  &   ---  &   ---  &   ---  &   ---  &   ---  &   --- \\
9   & 9.000 & 10.000 &   ---  &   ---  &   ---  &   ---  &   ---  &   ---  &   ---  &   ---  &   --- \\
\bottomrule
\end{tabular}
}
\end{table}
The behavior changes if just $f$ and not its cascadic interpolation is
used.
Table~\ref{Ex2:Jumps} shows for the methods $\Q{9}{k}$-$\vtd{9}{k}$,
$k=0,\dots,9$, the experimental convergence order of $\|(PP_s e)'\|_{L^2}$
after $s$ postprocessing steps based on jumps where at most $r+1-k= 10-k$ steps
have been carried out. The column $s=1$ shows, as predicted by our
theory, that the convergence order increases by $1$ for all methods.
The behavior using at least two postprocessing steps depends strongly on
the parameter $k$ of the variational time discretizations. For dG-like
methods (characterized by even $k$), an additional improvement by one order
is obtained independent of the number of postprocessing steps. The situation
is completely different for cGP-like method (corresponding to odd $k$). For
$k\equiv 3\mod 4$, the second postprocessing step does not lead to an
improvement of the convergence order compared to a single postprocessing step. If
$k\equiv 1\pmod 4$ then the second postprocessing step provides an
increased convergence order. However, for all cGP-like methods, the obtained
convergence rates start to decrease with increasing numbers of
postprocessing steps. This is in complete contrast to dG-like methods.
Calculations for the methods $\Q{10}{k}$-$\vtd{10}{k}$, $k=0,\dots,10$,
show for dG-like methods the same behavior as in the case $r=9$. However, 
the roles of $k\equiv 1\mod 4$ and $k\equiv 3\mod 4$ for cGP-like methods
are switched compared to the case $r=9$.

\begin{table}[htb!]
\caption{Example~\ref{Ex2}: Experimental orders of convergence for $\|(PP_s e)'\|_{L^2}$
using $\Q{9}{k}$-$\vtd{9}{k}$, $k=0,\dots,9$, and $s$ postprocessing steps based
on residuals, cf.~Theorem~\ref{th:postproc}.}
\label{Ex2:Residual}
\centerline{%
\begin{tabular}{c@{\hspace*{0.9em}}c@{\hspace*{0.9em}}c@{\hspace*{0.9em}}c@{\hspace*{0.9em}}
c@{\hspace*{0.9em}}c@{\hspace*{0.9em}}c@{\hspace*{0.9em}}c@{\hspace*{0.9em}}c@{\hspace*{0.9em}}
c@{\hspace*{0.9em}}c@{\hspace*{0.9em}}c}
\toprule
$k$ & $s=0$  & $s=1$  & $s=2$  & $s=3$  & $s=4$  & $s=5$  & $s=6$  & $s=7$  & $s=8$  & $s=9$  & $s=10$\\
\midrule
0   & 9.001 & 10.000 & 11.001 & 11.001 & 11.001 & 11.001 & 11.001 & 11.001 & 11.001 & 11.001 & 11.001\\
1   & 9.000 & 10.000 & 11.000 & 11.001 & 11.000 & 11.000 & 11.000 & 11.000 & 11.000 & 11.000 &   --- \\
2   & 9.001 & 10.001 & 11.001 & 11.002 & 11.001 & 11.001 & 11.001 & 11.001 & 11.001 &   ---  &   --- \\
3   & 9.000 & 10.000 & 11.000 & 11.001 & 11.000 & 11.000 & 11.000 & 11.000 &   ---  &   ---  &   --- \\
4   & 9.001 & 10.001 & 11.001 & 11.002 & 11.001 & 11.001 & 11.001 &   ---  &   ---  &   ---  &   --- \\
5   & 9.000 & 10.000 & 11.000 & 11.002 & 11.000 & 11.000 &   ---  &   ---  &   ---  &   ---  &   --- \\
6   & 9.001 & 10.001 & 11.001 & 11.003 & 11.002 &   ---  &   ---  &   ---  &   ---  &   ---  &   --- \\
7   & 9.000 & 10.000 & 11.000 & 11.002 &   ---  &   ---  &   ---  &   ---  &   ---  &   ---  &   --- \\
8   & 9.001 & 10.001 & 11.002 &   ---  &   ---  &   ---  &   ---  &   ---  &   ---  &   ---  &   --- \\
9   & 9.000 & 10.000 &   ---  &   ---  &   ---  &   ---  &   ---  &   ---  &   ---  &   ---  &   --- \\
\bottomrule
\end{tabular}
}
\end{table}
Our theory provides that postprocessing based on jumps and postprocessing
based on residuals are equivalent if a single postprocessing step is
applied. The situation changes if at least two postprocessing steps are
used. Table~\ref{Ex2:Residual} shows the experimental orders of convergence
of $\|(PP_s e)'\|_{L^2}$ after $s$ postprocessing steps based on residuals
for the methods $\Q{9}{k}$-$\vtd{9}{k}$, $k=0,\dots,9$, that are the same
ones as used for obtaining the results in Table~\ref{Ex2:Jumps}.
Independent of $k$, the application of at least two postprocessing steps
leads always to an improvement of the convergence order by two compared to
the results without postprocessing. Moreover, the orders of convergence do
not decrease even if more than two postprocessing steps based on residuals
are applied. The same behavior is observe for the methods
$\Q{10}{k}$-$\vtd{10}{k}$, $k=0,\dots,10$.

\section*{Acknowledgement}
This is a preprint of an article published in BIT Numerical Mathematics. The final
authenticated version is available online at: \url{https://doi.org/10.1007/s10543-021-00851-6}

  \begin{appendix}
  
\section{Direct proof for the alternative postprocessing}
\label{app:postAlt}
We now want to give a direct proof of Corollary~\ref{cor:postprocAlt}.
Similar to the proof of Theorem~\ref{th:postproc} we shall verify that
$\widetilde{U}$ satisfies all conditions for $\Q{r}{k}$-$\vtd{r+1}{k+2}$ where
$\Q{r}{k}$ is the quadrature rule associated to $\vtd{r}{k}$ which is
exact for polynomials of degree less than or equal to $2r-k$.

Since $\tilde{\vartheta}_n$ merely is a multiple of $\vartheta_n$, analogously
to \eqref{help:quadZeroPP} we have
\begin{gather}
\label{help:quadZeroPPtilde}
\Q{r}{k}\big[\tilde{\vartheta}_n' \varphi\big]
= -\delta_{0,k} (\tilde{\vartheta}_n \varphi)(t_{n-1}^+) \qquad \forall \varphi \in P_{r-k}(I_n,\RR).
\end{gather}
We will show that $\widetilde{U}$ satisfies all conditions for $\Q{r}{k}$-$\vtd{r+1}{k+2}$.
\begin{enumerate}[(a)]
\item \label{enu:123app}
Similar to the proof of Theorem~\ref{th:postproc} we prove the initial condition
$\widetilde{U}(t_{n-1}^+) = \widetilde{U}(t_{n-1}^-)$, the point conditions at
$t_n^-$ up to order $\left\lfloor\tfrac{k}{2}\right\rfloor -1$ and at $t_{n-1}^+$ up to order
$\left\lfloor\tfrac{k-1}{2}\right\rfloor -1$, as well as the variational condition.
Note that for $k=0$ the proof of the variational condition is even
easier since the initial condition $\widetilde{U}(t_{n-1}^+) = \widetilde{U}(t_{n-1}^-)$
is immediately clear from the alternative definition of the postprocessing. In detail, we have

\begin{enumerate}[({a}1)]
\item \label{enu:1app} \underline{Conditions at $t_n^-$ for
$0 \leq i \leq \left\lfloor\tfrac{k+2}{2}\right\rfloor-2 = \left\lfloor\tfrac{k}{2}\right\rfloor -1$:} \\
We obtain from the definitions of $\widetilde{U}$ and $U$
\begin{align*}
M \widetilde{U}^{(i+1)}(t_n^-)
= \frac{\d^i}{\d t^i} F\big(t,\widetilde{U}(t)\big) \Big|_{t=t_n^-}.
\end{align*}
\item \label{enu:2app} \underline{Conditions at $t_{n-1}^+$ for
$0 \leq i \leq \left\lfloor\tfrac{k+2-1}{2}\right\rfloor-2 = \left\lfloor\tfrac{k-1}{2}\right\rfloor -1$:} \\
We obtain from the definitions of $\widetilde{U}$ and $U$
\begin{align*}
M \widetilde{U}^{(i+1)}(t_{n-1}^+)
= \frac{\d^i}{\d t^i} F\big(t,\widetilde{U}(t)\big) \Big|_{t=t_{n-1}^+}\!.
\end{align*}
\item \label{enu:3app} \underline{Variational condition:} \\
Using the identity~\eqref{help:quadZeroPPtilde} and $\big[\widetilde{U}\big]_{n-1}=0$
we gain
\begin{align*}
\Q{r}{k}\Big[\big(M \widetilde{U}',\varphi\big)\Big]
= \Q{r}{k}\Big[\big(F(\cdot,\widetilde{U}(\cdot)),\varphi\big)\Big]
\qquad \qquad \forall \varphi \in P_{r-k}(I_n,\RR^d).
\end{align*}
\end{enumerate}
\end{enumerate}
It remains to verify the two remaining point conditions.
\begin{enumerate}[(a)]
\setcounter{enumi}{1}
\item \label{enu:4app}
\underline{Condition at $t_{n-1}^+$ for $i = \left\lfloor\tfrac{k+2-1}{2}\right\rfloor-1 = \left\lfloor\tfrac{k-1}{2}\right\rfloor$, if $k\geq 1$:} \\
First of all, let $n=1$. Then by the initial condition and~\eqref{enu:2app} we
have for $0 \leq i \leq \left\lfloor \frac{k-1}{2} \right\rfloor -1$
\begin{gather*}
\widetilde{U}(t_0^+) = u_0, \qquad
M \widetilde{U}^{(i+1)}(t_0^+)
= \frac{\d^i}{\d t^i} F\big(t,\widetilde{U}(t)\big) \Big|_{t=t_0^+}.
\end{gather*}
Recalling the definition of $u^{(i)}(t_0)$, we iteratively obtain
\begin{gather*}
\widetilde{U}^{(i)}(t_0^+) = u^{(i)}(t_0), \qquad 0 \leq i \leq \left\lfloor \tfrac{k-1}{2} \right\rfloor.
\end{gather*}
From this we conclude, using the definitions of $\widetilde{U}$, $\tilde{a}_n$, and $\tilde{\vartheta}_n$, that
\begin{align*}
M \widetilde{U}^{\left(\left\lfloor \frac{k-1}{2} \right\rfloor +1\right)}(t_0^+)
  & = M U^{\left(\left\lfloor \frac{k-1}{2} \right\rfloor +1\right)}(t_0^+)
  - M \tilde{a}_1 \underbrace{\tilde{\vartheta}_1^{\left(\left\lfloor \frac{k-1}{2} \right\rfloor +1\right)}(t_0^+)}_{=1}
= M u^{\left(\left\lfloor \frac{k-1}{2} \right\rfloor +1\right)}(t_0) \\
  & = \frac{\d^{\left\lfloor \frac{k-1}{2} \right\rfloor}}{\d t^{\left\lfloor \frac{k-1}{2} \right\rfloor}}
  F\big(t,u(t)\big) \big|_{t=t_0}
= \frac{\d^{\left\lfloor \frac{k-1}{2} \right\rfloor}}{\d t^{\left\lfloor \frac{k-1}{2} \right\rfloor}}
F\big(t,\widetilde{U}(t)\big) \big|_{t=t_0^+}.
\end{align*}

Now, let $n > 1$. We assume that $\widetilde{U}$ solves the $I_{n-1}$-problem
which will be finally shown when also the last condition is proved, see~\eqref{enu:5app}.
Then by construction
\begin{align*}
& M \widetilde{U}^{\left(\left\lfloor \frac{k-1}{2} \right\rfloor+1\right)}(t_{n-1}^+)
= M U^{\left(\left\lfloor \frac{k-1}{2} \right\rfloor+1\right)}(t_{n-1}^+)
- M \tilde{a}_n \underbrace{\tilde{\vartheta}_n^{\left(\left\lfloor \frac{k-1}{2} \right\rfloor+1\right)}(t_{n-1}^+)}_{=1}\\
& \quad = M \widetilde{U}^{\left(\left\lfloor \frac{k-1}{2} \right\rfloor+1\right)}(t_{n-1}^-)
= \frac{\d^{\left\lfloor \frac{k-1}{2} \right\rfloor}}{\d t^{\left\lfloor \frac{k-1}{2} \right\rfloor}}
F\big(t,\widetilde{U}(t)\big) \Big|_{t=t_{n-1}^-}
= \frac{\d^{\left\lfloor \frac{k-1}{2} \right\rfloor}}{\d t^{\left\lfloor \frac{k-1}{2} \right\rfloor}}
  F\big(t,\widetilde{U}(t)\big) \Big|_{t=t_{n-1}^+}
\end{align*}
where we also used that we already know that $\widetilde{U}^{(i)}(t_{n-1}^+) = \widetilde{U}^{(i)}(t_{n-1}^-)$
for $0 \leq i \leq \left\lfloor \frac{k-1}{2} \right\rfloor$.
\item \label{enu:5app}
\underline{Condition at $t_n^-$ for $i = \left\lfloor\tfrac{k+2}{2}\right\rfloor-1 = \left\lfloor\tfrac{k}{2}\right\rfloor$:} \\
We have to show that
\begin{gather*}
M \widetilde{U}^{\left(\left\lfloor\frac{k}{2}\right\rfloor+1\right)}(t_{n}^-)
= \frac{\d^{\left\lfloor \frac{k}{2} \right\rfloor}}{\d t^{\left\lfloor \frac{k}{2} \right\rfloor}} F\big(t,\widetilde{U}(t)\big) \Big|_{t=t_{n}^-}.
\end{gather*}
This can be done similar to the proof of Theorem~\ref{th:postproc}.

The variational condition for $\widetilde{U}$ is used with the special test functions
$\widetilde{\varphi}_j \in P_{r-k}(I_n,\RR^d)$, $j=1, \ldots, d$, that vanish at all inner
quadrature points, i.e.,
\begin{gather*}
\widetilde{\varphi}_j(t_{n,i}) = 0,\quad i = 1,\ldots,r-k,
\qquad \qquad \text{and satisfy} \qquad \qquad
\widetilde{\varphi}_j(t_n^-) = e_j.
\end{gather*}
By~\eqref{enu:3app} we have
\begin{gather*}
\Q{r}{k}\Big[\big(M \widetilde{U}',\widetilde{\varphi}_j\big)\Big]
= \Q{r}{k}\Big[\big(F(\cdot,\widetilde{U}(\cdot)),\widetilde{\varphi}_j\big)\Big],
\qquad j = 1,\ldots,d.
\end{gather*}
The special definition of $\widetilde{\varphi}_j$, the definition of the quadrature rule, and the already known
identities from~\eqref{enu:1app}, \eqref{enu:2app}, and~\eqref{enu:4app} (for $t_{n-1}^+$ and $k \geq 1$) yield after a short calculation
using Leibniz' rule for the $i$th derivative that
\begin{align*}
&& \Q{r}{k}\Big[\big(M \widetilde{U}',\widetilde{\varphi}_j\big)\Big]
& = \Q{r}{k}\Big[\big(F(\cdot,\widetilde{U}(\cdot)),\widetilde{\varphi}_j\big)\Big], \quad j=1,\ldots,d, \\
\Leftrightarrow
&& w_{\left\lfloor\frac{k}{2}\right\rfloor}^R
M \widetilde{U}^{\left(\left\lfloor\frac{k}{2}\right\rfloor+1\right)}(t_n^-) \cdot
\underbrace{\widetilde{\varphi}_j(t_n^-)}_{=e_j}
& =w_{\left\lfloor\frac{k}{2}\right\rfloor}^R
\frac{\d^{\left\lfloor \frac{k}{2} \right\rfloor}}{\d t^{\left\lfloor \frac{k}{2} \right\rfloor}} F\big(t,\widetilde{U}(t)\big)
\Big|_{t=t_n^-} \!\!\cdot \underbrace{\widetilde{\varphi}_j(t_n^-)}_{=e_j},
\quad j=1,\ldots,d, \\
\Leftrightarrow
&& M \widetilde{U}^{\left(\left\lfloor\frac{k}{2}\right\rfloor+1\right)}(t_{n}^-)
& =\frac{\d^{\left\lfloor \frac{k}{2} \right\rfloor}}{\d t^{\left\lfloor \frac{k}{2} \right\rfloor}} F\big(t,\widetilde{U}(t)\big)
\Big|_{t=t_{n}^-}
\end{align*}
where we exploited that $w_{\left\lfloor\frac{k}{2}\right\rfloor}^R \neq
0$. Note that~\eqref{enu:4app} is only needed for $t_{n-1}^+$ and already
completely proven for $t_0^+$. Hence, \eqref{enu:4app} and~\eqref{enu:5app}
can be iteratively shown for all $n$.
\end{enumerate}
Hence, $\widetilde{U}$ solves $\Q{r}{k}$-$\vtd{r+1}{k+2}$.

\end{appendix}

\bibliographystyle{plain}
\bibliography{vtd_odeBiblio}

\end{document}